\documentclass[a4paper,UKenglish,cleveref, autoref, thm-restate]{lipics-v2021}

\pdfoutput=1 
\hideLIPIcs  

\graphicspath{{./figures/}}

\usepackage{xspace}
\usepackage{dsfont}

\newcommand{\mean}{\mathds{E}}
\newcommand{\bigO}{\mathcal{O}}

\newcommand{\mF}{\mathcal{F}}
\newcommand{\mG}{\mathcal{G}}

\newcommand*{\eg}{\textit{e.g.}\@\xspace}

\newcommand{\laplace}{\mathcal{L}}

\newcommand{\ex}{\xi}
\newcommand{\cat}{\operatorname{Cat}}

\newtheorem*{lemma*}{Lemma}
\newtheorem*{proposition*}{Proposition}
\newtheorem*{theorem*}{Theorem}

\title{Tree walks and the spectrum of random graphs} 

\author{Eva-Maria Hainzl}
{TU Wien, Austria}
{}
{}
{This work was partially founded by the Austrian Science
Foundation FWF, projects F50-02, F55-02.}

\author{\'Elie {de Panafieu}}
{Nokia Bell Labs, France}
{}
{https://orcid.org/0009-0002-1386-971X}
{}

\authorrunning{E-M. Hainzl and \'E. de Panafieu} 

\Copyright{Eva-Maria Hainzl and \'Elie de Panafieu} 



\category{} 

\relatedversion{A short version of this paper was presented at the 35th International Conference on Probabilistic, Combinatorial and Asymptotic Methods for the Analysis of Algorithms (AofA2024).} 

\funding{This work was supported by the RandNET project, MSCA-RISE - Marie Sk\l{}odowska-Curie Research and Innovation Staff Exchange Programme (RISE), Grant agreement 101007705.}

\acknowledgements{We thank Nicolas Curien and Laurent Ménard for encouraging us to work on this topic, their availability and their insights on the spectrum of random graphs. We also thank the Institut de Recherche en Informatique Fondamentale (IRIF),  Universit\'e Paris Cit\'e, for hosting us.}

\nolinenumbers 

\begin{document}

\maketitle

\begin{abstract}
It is a classic result in spectral theory that the limit distribution of the spectral measure of random graphs $G(n,p)$ converges to the semicircle law in case $np$ tends to infinity with $n$. The spectral measure for random graphs $G(n,c/n)$ however is less understood. In this work, we combine and extend two combinatorial approaches by Bauer and Golinelli (2001) and Enriquez and Menard (2016) and approximate the moments of the spectral measure by counting walks that span trees.
\end{abstract}

        \section{Introduction}

Random matrix theory studies the spectrum of random matrices
and has found many applications, including in
physics \cite{wigner1967random},
wireless communication \cite{tulino2004random}
and numerical analysis \cite{edelman2005random}.
A fundamental result of this field is
that the limit distribution of the spectral measure
of so-called Wigner matrices converges to the semicircle law~\cite{wigner1955characteristic, wigner1958distribution} and it is worth mentioning that a common proof of this theorem by the moment method relies on counting closed walks on trees (\eg \cite{furedi1981eigenvalues}).
This universal law has been extended to several other classes,
such as adjacency matrices of random regular graphs 
\cite{mckay1981expected,tran2013sparse}
and Erd\H{o}s-R\'enyi random graphs $G(n,p)$ when $pn \rightarrow \infty$. In particular, Bauer and Golinelli~\cite{physics} pointed out the importance of the spectral measure of adjacency matrices of random graphs and explained how to compute the moments by counting walks on trees. Zakharevich~\cite{inna_z} picked up on the approach and showed further that the spectral distribution of $G(n,c/n)$ converges to a limit distribution $\mu^c$ which has infinite support. However, for $p=c/n$, several technical conditions of classic theorems in probability theory are not met
such that one could apply standard techniques and despite recent progress \cite{spec1,spec2,spec3,spec4,spec5}, $\mu^c$ remains an enigma.
In \cite{Bushygraphs}, Enriquez and Ménard returned to combinatorial methods and computed several terms of the asymptotic expansion, as $c$ tends to infinity, of the moments of the normalized spectral measures
\[
    \mu_n^c =
    \frac{1}{n}\sum_{\lambda \in Sp(c^{-1/2}A(G(n,c/n)))}\delta_\lambda
\]
where $A(G(n,c/n))$ is the adjacency matrix of a random graph $G(n,c/n)$.
We go along the steps in the computation of moments of this measure for clarity and start with 
\[
    m_{\ell}(\mu^c_n) = \sum_{G} \mathbb{P}\left[G\left(n,\frac{c}{n}\right) = G\right] \cdot \frac{1}{n}\sum_{\lambda \in Sp\left(c ^{-1/2} A(G)\right) }\lambda^{\ell}.
\] 
This formulation reduces to counting closed walks in $G$, since the sum of the eigenvalues to the power $\ell$ is just the trace of the matrix to the power $\ell$,
and a value $\left(A^\ell\right)_{i,i}$ on the diagonal of this matrix is the number of closed walks of length $\ell$ starting at the vertex $i$. That is,
\[
    \sum_{\lambda \in Sp\left(c ^{-1/2} A(G)\right)} \lambda^{\ell} = \left(\frac{1}{c }\right)^{\ell/2}\mbox{tr}\left(A(G)^{\ell}\right) = \sum_{\text{closed walk }(v_1,v_2,\dots v_{\ell}, v_1)\in G} \left(\frac{1}{c }\right)^{\ell/2}.
\]
Thus, the moment equals
\[
    m_{\ell}(\mu^c_n) =
    \frac{1}{n}\frac{1}{c^{\ell/2}}
    \sum_{(v_1,v_2,\dots v_{\ell})\in [n]^{\ell}}
    \mean \left[X_{v_1,v_2}X_{v_2,v_3}\cdots X_{v_{\ell},v_1} \right],
\]
where $X_{v_i,v_j}$ is the random variable taking the value $1$ if the edge $(v_i,v_j)$ is in the graph and $0$ otherwise.
Observe that if a closed walk $(v_1, \ldots, v_{\ell}, v_1)$
contains $e$ distinct edges, then
$\mean \left[X_{v_1,v_2}X_{v_2,v_3}\cdots X_{v_{\ell},v_1} \right] = (c/n)^e$.
The number of closed walks on $[n]$ of length $\ell$
with $m$ vertices is bounded by $n^m m^{\ell}$.
Since the total number of vertices is bounded by the length,
we have $n^m m^{\ell} \leq n^m \ell^{\ell}$.
The contribution to the moment
of all such closed walks
containing $e$ distinct edges is bounded by
\[
    \frac{1}{n} \frac{1}{c^{\ell/2}}
    n^m \ell^{\ell}
    \left( \frac{c}{n} \right)^e
    =
    c^{e - \ell/2}
    \ell^{\ell}
    n^{m - e - 1}.
\]
We are considering a fixed moment $\ell$,
so this tends to $0$ with $n$ whenever $m < e + 1$, that is,
whenever the graph (necessarily connected)
induced by the closed walk is not a tree.
In particular, when $\ell$ is odd,
the induced graph cannot be a tree,
so the moment of order $\ell$ tends to $0$.

Let $w_{m, 2 \ell}$ denote the number of closed walks
of length $2 \ell$ spanning a tree with $m$ vertices.
We now consider the even moment of order $2 \ell$
and split the sum according to the number $m$
of distinct vertices in the closed walk
\[
    m_{2 \ell}(\mu^c_n) =
    \frac{1}{n} \frac{1}{c^{\ell}}
    \sum_{m = 1}^{\ell + 1}
    \binom{n}{m}
    \left( \frac{c}{n} \right)^{m-1}
    w_{m, 2 \ell}.
\]
Let us define the limit distribution
$\mu^c = \lim_{n \to +\infty} \mu^c_n$.
Then its odd moments are zero
and its moment of order $2 \ell$ is
\[
    m_{2 \ell}(\mu^c)
    =
    \lim_{n \to +\infty}
    \frac{1}{n} \frac{1}{c^{\ell}}
    \sum_{m = 1}^{\ell + 1}
    \binom{n}{m}
    \left( \frac{c}{n} \right)^{m-1}
    w_{m, 2 \ell}
    =
    \sum_{m = 1}^{\ell + 1}
    \frac{1}{c^{\ell - m + 1}}
    \frac{w_{m, 2 \ell}}{m!}.
\]
By identifying the generating functions  of $(w_{m,2\ell})_{\ell \geq 0}$, for  $m=\ell+1$ and $m = \ell$, as the Stieltjes transform of a specific measure, Enriquez and M\'enard were able to derive an approximation of the moments of the limit law and computational experiments showed that even the density of this measure approximated the shape of the histograms of eigenvalues of sampled matrices quite well. An extension of this approximation to the order $c^{-2}$ took considerable effort on several sides, including the combinatorics of closed walks on trees.

The aim of this paper is to provide further insight into what we call \emph{tree walks}, and consequently an efficient way to compute the numbers $w_{m,2\ell}$, for all $2\ell \geq 0$ and $0\leq m \leq \ell+1$, and their generating functions.
But as we delve further into their connection with the spectral measure, we come across surprising and beautiful identities involving the generating function of the Catalan numbers.

\cref{sec:main_results} presents the formal definition
of various tree walk families and our main results,
which are \cref{thm:strucT_k} and \cref{thm:expand}.
\cref{thm:strucT_k} expresses the generating function
of tree walks as a rational function
of the Catalan generating function.
\cref{thm:expand} gives several error terms
for an asymptotic approximation of $\mu^c$
as $c$ tends to infinity.
We also presents \cref{conj},
which states that this asymptotic approximation
could be extended to an arbitrary order,
turning it into a form of asymptotic expansion.
The proof of \cref{thm:strucT_k}
and \cref{thm:expand}
are given respectively in \cref{sec:decomp,sec:cons}.
Numerical experiments are provided in \cref{sec:computational}.

    \section{Main results} \label{sec:main_results}

Before we state our main results, let us clarify some definitions.

\begin{definition}[Tree walks]
A \emph{tree walk} of size $m$
is a walk on the complete labeled graph of size $m$ that
visits every node, starts and ends at the same node, and induces a tree.
More formally, a tree walk $W = (v_1,v_2,\dots, v_{\ell})$ is a sequence of $v_i \in [m]$ such that 
\[
    V:= \bigcup_{j \in [\ell]} \{v_j\},\quad E:= \bigcup_{j \in [\ell-1]}\{(v_j,v_{j+1})\}\cup \{(v_{\ell},v_1)\}
\]
define a labelled tree $T(W)$ with vertex set $V = [m]$ and edge set $E$. Further, we define $v_1$ to be the root of the induced tree $T(W)$. Thus, we talk freely about the root and leaves of $W$, when referring to the root and leaves of $T(W)$. We further stick to the convention that if the root has degree $1$ it is also a leaf of $T(W)$. \\
In the following, we study the number $w_{m,2\ell}$ of tree walks of length $2\ell$ that span a tree of size $m$ and the generating function
\[
    W(v,z) = \sum_{\ell, m \geq 0} w_{m,2\ell} \frac{v^m}{m!} z^{\ell}. 
\]
Since a walk of length $2 \ell$ spans a tree with at most $\ell+1$ vertices, we have $w_{m,2\ell} = 0$ if $\ell<m-1$  and for $\ell \geq 1$, we define $w_{0,2\ell}=w_{1,2\ell}=0$ and $w_{1,0}=1$, $w_{0,0}=0$.
\end{definition}

The ordinary generating function of the moments of $\mu^c$ is therefore given by
\begin{equation}\label{eq:ordGF} 
M_{\mu^c}(z) 
= \sum_{\ell \geq 0} m_{2\ell}(\mu^c)z^{2\ell} 
= \sum_{\ell \geq 0} \sum_{m = 0}^{\ell +1}w_{m,2\ell}\frac{c^m}{m!}c^{-\ell-1}z^{2\ell} 
= \frac{1}{c} \, W\left(c , \frac{z^2}{c}\right),
\end{equation}
where $m_0(\mu^c)=1$ as always.
However, we could have restructured $M_{\mu^c}(z)$ like Enriquez and Ménard in~\cite{Bushygraphs} as well. We just sum over the negative exponent $\xi = \ell+1-m$ of $c$ such that
\begin{equation}\label{eq:2nd_exp} M_{\mu^c}(z) = \sum_{\ell \geq 0} m_{2\ell}(\mu^c)z^{2\ell} = \sum_{\ell \geq 0} \sum_{m\geq 0} w_{m,2\ell}\frac{c^m}{m!}c^{-\ell-1}z^{2\ell} = \sum_{\ell \geq 0} \sum_{\xi \geq 0} \frac{w_{\ell-\xi+1,2\ell}}{(\ell-\xi+1)!} \frac{z^{2\ell}}{c^\xi}.
\end{equation}
This expansion in turn motivates the following definition. 

\begin{definition}[Excess of a tree walk] \label{def:excess}
    If an edge is traversed $2k$ times in a tree walk, then the \emph{excess} of the edge $e$ is defined as $\ex(e) = k-1$. 
    The \emph{excess} of a tree walk $W$ is the sum over the excess of all edges in the induced tree $T(W) = (V,E)$. Hence, it is half its length minus the number of edges of the tree, $\ex(W)\;= \;\ell - |E|$.
    An edge with positive excess is called an \emph{excess edge} and an edge without excess a \emph{simple edge}. We denote the generating function of tree walks with excess $\xi$ by
\[
    W_\xi(z) =
    \sum_{\ell \geq 0}
    \frac{w_{\ell - \xi + 1, 2 \ell}}{(\ell - \xi + 1)!} \, z^{\ell},
\]
where $w_{m, 2 \ell}$ is the number of tree walks of length $2\ell$ that span a tree of size $m$.
\end{definition}

Thus, the relation between the generating functions we have defined so far is
\[
    M_{\mu^c}(z) = \frac{1}{c} W\left(c,\frac{z^2}{c}\right) = \sum_{\xi\geq 0} \frac{1}{c^\xi} W_\xi\left(z^2\right).
\]

Bauer and Golinelli~\cite{physics} introduced
in the sequence $w_{m,2\ell}$
an additional parameter $d$
counting the number of times the walk
leaves the root.
This approach allowed them to compute the values $w_{m,2\ell}$
for $2\ell$ and $m$ up to $120$ \cite{seq},
and they conjectured a particular form for $w_{m,2\ell}$
that we prove in the next theorem.
When we translate this decomposition in generating functions,
an equation for $W(x,v,z)$ is obtained,
where $x$ marks the parameter $d$.
Unfortunately, this equation is not particularly amenable to classic analysis with complex analytic methods as it involves a Laplace transform.
Our approach on the other hand is reminiscent of the decomposition of graphs with given excess by Wright \cite{wright}. Not only do we prove a well founded recursion in~$z$ and~$v$, but we provide more insight into the structure of tree walks and their generating function. Most importantly, we compute closed expressions for $w_{m,2\ell}, \ell \geq 0$ for fixed (small) $m$ and prove a conjecture from~\cite{physics}.

\begin{theorem} \label{thm:strucT_k}
Let $C(z) = \frac{1 - \sqrt{1 - 4 z}}{2 z}$ denote
the generating function of the Catalan numbers,
and $W_\xi(z)$
denote the generating function of tree walks of excess $\xi$
from Definition~\ref{def:excess}.
Then $W_0(z) = C\left(z\right)$ and for any $\xi \geq 1$,
there are polynomials $(K_{\xi,s}(x))_{0 \leq s \leq 2\xi-2}$
with non-negative coefficients of degree $2\xi+s$
such that
\[
    W_\xi(z) =
    C\left(z\right)
    \sum_{s=0}^{2\xi-2}
    \frac{K_{\xi,s}\left(zC(z)^2\right)}
    {\big(1-zC(z)^2\big)^{s+1}}.
\]
In particular, denoting by $\cat(n)$ the $n$-th Catalan number,
we have
\[
    K_{\xi,2\xi-2}(x) = \cat(\xi-1)x^{4\xi-2}
    \qquad \text{and} \qquad
    K_{\xi,2\xi-3}(x) = (3\xi-1) \cat(\xi-1)x^{4\xi-3}.
\]
\end{theorem}

We establish a recursion for the polynomials $K_{\xi,s}(x)$ in Section~\ref{sec:decomp}.
This enables the successive computation of three quantities. First, the generating function $W_\xi(z)$ for any $\xi$, then the series $M_{\mu^c}(z)$ up to an arbitrary degree in $c$, given sufficient computational power, and finally the moments $m_{2\ell}$.

Our next theorem significantly extends Theorem 3 from~\cite{Bushygraphs}.
It approximates $\mu^c$ for large $c$.
There are many notions of convergence for measures.
The one we consider here is the convergence
of all moments
(restricting to the even ones since the odd ones vanish).
Further, when looking at the limit
of a sequence of random variables,
it is common to rescale them by their mean and standard deviation.
Here, the rescaling takes the form of a dilation operator
$\Lambda_\alpha$, for $\alpha > 0$.
This operator transforms a measure $\mu$
into the measure $\Lambda_\alpha(\mu)$
satisfying for every Borel set $A$
\[
    \Lambda_{\alpha}(\mu)(A) = \mu(A/\alpha).
\]

\begin{theorem}\label{thm:expand}
Let $m_{\ell}(\mu)$ denote the $\ell$-th moment of a measure $\mu$
and $\Lambda_\alpha$ the dilation operator
defined above.
Then as $c \to \infty$, it holds for all $\ell \geq 0$ that
\[
    m_{2 \ell}(\mu^c) =
    m_{2 \ell} \bigg(
    \Lambda_{f(1/c)} \bigg(
    \sigma + \sum_{i=1}^5\frac{1}{c^i}\sigma_i
    \bigg) \bigg)
    + \bigO \left(\frac{1}{c^6}\right)
\]
where $f(1/c) = 1+\frac{1}{2c} + \frac{3}{8c^2}+\frac{29}{16c^3}+\frac{1987}{128c^4}+\frac{47247}{256c^5}$, $\sigma$ is the semicircle law and all $\sigma_1, \sigma_2, \dots \sigma_5$ are signed measures explicitly given in Section~\ref{sec:computational}, with total mass $0$.
\end{theorem}

This approximation entails some curious identities concerning the generating functions $W_\xi(z)$ and prompts us to state the following conjecture which we discuss in more detail in Section~\ref{sec:cons}.

\begin{conjecture}\label{conj}
    Let $M_{\mu^c}(z)$ be the ordinary moment generating function of $\mu^c$ as defined in~(\ref{eq:ordGF}). Then there exists a unique power series $P(x)$ with non-negative integer coefficients such that all $V_i(z)$ which are given by
    \[
        V_i(z) := \left[c^{-i}\right] M_{\mu^c}\left(\sqrt{\frac{z}{P(1/c)}}\right),\quad i\geq 0
    \]
    are the product of $C\left(z\right)$ and a polynomial in $zC\left(z\right)^2$.
\end{conjecture}

Let us denote by $f_k(x)$ the truncation of order $k$ of $\sqrt{P(x)}$.
If the previous conjecture holds, there exist
signed measures $\sigma_1, \ldots, \sigma_k$
explicitly computable from $V_1(z), \ldots, V_k(z)$
such that for any $\ell$,
the moment of order $2 \ell$ of $\mu^c$ is
\[
    m_{2 \ell}(\mu^c)
    =
    m_{2 \ell} \bigg(
    \Lambda_{f_k(1/c)} \bigg(
        \sigma + \sum_{i=1}^k c^{-i} \sigma_i
    \bigg) \bigg)
    + \bigO(c^{-k-1}).
\]
Thus, Conjecture~\ref{conj} provides a form of asymptotic expansion
for $\mu^c$ as $c$ tends to infinity.

    \section{Decomposition of tree walks}\label{sec:decomp}

Our proof of Theorem \ref{thm:strucT_k} involves reducing a tree walk with excess $\xi$ by most of its simple edges to its \emph{kernel walk} and subsequently reversing the contraction by blowing it up to an arbitrary tree walk with excess $\xi$. The following subsection is focused on this decomposition process and the subsequent subsection on the proof of Theorem \ref{thm:strucT_k} and a recursion enumerating kernel walks.

    \subsection{Kernel walks}
    \label{sec:kernel:walks}

Recall that an edge of a tree walk $W$ is \emph{simple}
if it is traversed exactly twice,
and is an \emph{excess edge} otherwise.

\begin{definition}[Kernel walks]\label{def:kernels} Given a tree walk $W$, we define the \emph{kernel} of the tree walk or simply the \emph{kernel walk} $W_K$ as the tree walk we obtain by the following procedure.
\begin{enumerate}
    \item Set $W'=W$ and let $T(W')$ be its induced tree.
    \item  While there exists a simple edge $e$ incident to a leaf in $T(W')$ which is not the root, delete both occurrences of $e$ in $W'$.
    \item While the root $u$ of the tree is a leaf and incident to a simple edge $\{u,v\}$, delete this edge in $W'$ and choose $v$ as the root of $T(W')$. 
    \item While there exists a vertex $v$ in $T(W')$ that is not the root and only incident to two simple edges $e_i=e_{j+1}=\{u,v\}$ and $e_{i+1}=e_j=\{v,w\}$, replace both consecutive pairs $e_i,e_{i+1}$ and $e_j,e_{j+1}$ with $\{u,w\}$ in $W'$.
    \item Set $W_K = W'$.
\end{enumerate}
Naturally, a tree walk $W$ with kernel $W_K=W$ is itself called a kernel walk.
Further, we define $k_{\xi,s,2\ell}$ to be the number of kernel walks of length $2\ell$ with excess $\xi$, where the induced tree has $s$ simple edges and we define the corresponding generating function
\[K_\xi(u,v,z) = \sum_{s,\ell \geq 0} k_{\xi,s,2\ell} \, u^s \frac{v^{\ell-\xi+1}}{(\ell-\xi+1)!}z^{\ell},\]
where $u$ counts the number of simple edges, $v$ the number of vertices in the induced tree and $z$ the half-length of the walk.
\end{definition}
This procedure is illustrated below.
Note that the variable $v$ in the generating function of kernel walks is superfluous since its exponent is fully determined by the length  and the excess of the walk. However, we choose to keep it to explain the factorial in the denominator. If we consider the generating function $K(u,v,z) = \sum_{\xi\geq 0} K_\xi(u,v,z)$, we can reconstruct the individual generating functions by
\[
    K_\xi(u,v,z) = \left[y^{\xi-1}\right] K\left(u,\frac{v}{y},yz\right).
\]

\begin{figure}[ht]
\raggedright \textbf{Example.} Reducing a tree walk $W$ and its induced tree $T(W)$ to its kernel. \vspace{4mm} \\
    \begin{subfigure}[b]{0.22\textwidth}
    \centering
    \includegraphics[width=\textwidth, page = 1]{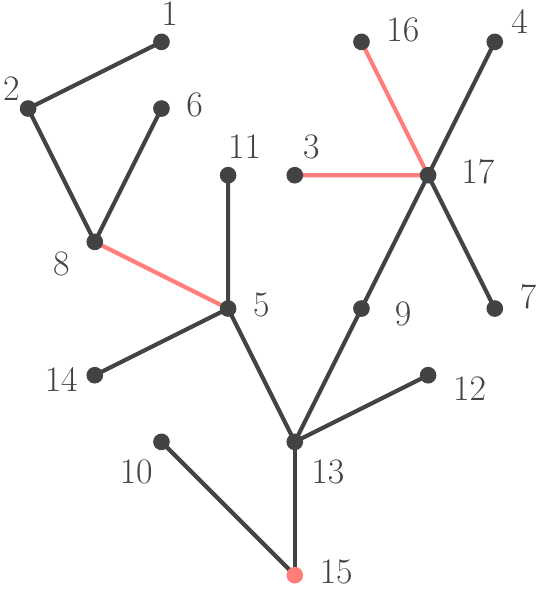}
    \caption{\footnotesize{Step 1: Set $T(W')$. \\Excess edges and the root in $T(W')$ are marked red.}}\label{fig:Step1}
    \end{subfigure}
    \hfill
    \begin{subfigure}[b]{0.22\textwidth}
    \centering
    \includegraphics[width=\textwidth, page = 2]{algorithm.pdf}
    \caption{ \footnotesize Step 2: Identify all leaves which are not incident to an excess edge}
    \end{subfigure}
    \hfill
    \begin{subfigure}[b]{0.22\textwidth}
    \centering
    \includegraphics[width=\textwidth, page = 3]{algorithm.pdf}
    \caption{\footnotesize Step 2: Remove blue vertices and update $T(W')$ by relabeling the vertices} \label{fig:Step2}
    \end{subfigure}\hfill 
    \hfill
    \begin{subfigure}[b]{0.22\textwidth}
    \centering
    \includegraphics[width=\textwidth, page = 4]{algorithm.pdf}\vspace{6.5mm}
    \caption{\footnotesize Repeat Step 2.}
    \end{subfigure}
    \\ \vspace{4mm}
    \hfill
    \begin{subfigure}[b]{0.22\textwidth}
    \centering
    \includegraphics[width=\textwidth, page = 6]{algorithm.pdf}\vspace{-2mm}
    \caption{\footnotesize Step 3: The root is a leaf}\label{fig:Step22}
    \end{subfigure}
    \hfill
    \begin{subfigure}[b]{0.22\textwidth}
    \centering
    \includegraphics[width=\textwidth, page = 7]{algorithm.pdf}\vspace{-2mm}
    \caption{\footnotesize Step 3: Choose new root and relabel vertices}\label{fig:Step23}
    \end{subfigure}
    \hfill
    \begin{subfigure}[b]{0.22\textwidth}
    \centering
    \includegraphics[width=\textwidth, page = 8]{algorithm.pdf}\vspace{-2mm}
    \caption{\footnotesize Step 4: Find adjacent simple edges}
    \end{subfigure}
    \hfill
    \begin{subfigure}[b]{0.22\textwidth}
    \centering
    \includegraphics[width=\textwidth, page = 9]{algorithm.pdf}\vspace{-5.5mm}
    \caption{\footnotesize Step 4: Update $T(W')$ by deleting vertex 4 and relabeling the vertices.}\label{fig:Step3}
    \end{subfigure}
    \label{fig:algo}
\end{figure}

\begin{lemma}\label{a:prop:char}
    Let $W$ be a kernel walk and $T(W)$ its induced tree. Then $T(W)$ is a rooted plane tree where
    \begin{enumerate}
        \item[(a)] every leaf in $T(W)$ is incident to an excess edge
        \item[(b)] every inner vertex has outdegree at least two or it is incident to an excess edge.
    \end{enumerate}
    Conversely, each rooted plane tree with edges colored by 'simple' or 'excess' which satisfies the conditions above appears as an induced tree of a kernel walk $W$.
\end{lemma}
\begin{proof}
    First, $T(W)$ is plane because the edges at a vertex $v$ are ordered by the order in which the walk traverses them for the first time. Further, by Step~2 and~3 in the contraction procedure, every leaf in the induced tree is incident to an excess edge. Also note that Step~4 does not change the degree of the vertices remaining in the updated induced tree such that we would not create a new leaf later in the procedure. Thus, Condition (a) has to hold.

    The second condition has to hold because otherwise either the root is incident to a single simple edge such that we would repeat Step~3 in the procedure or there exists a vertex $v$ which is not the root and incident one incoming one outgoing simple edge. However, in this latter case, we would have repeated Step~4 in the procedure one more time. 

    Now let us assume we are given a rooted plane tree $T$ with edges colored 'simple' or 'excess' that satisfies Conditions (a) and (b). We construct a kernel walk $W$ from the contour walk $W' = \{v_1,v_2,\dots,v_1\}$ of $T$ by inserting the steps $v_i,v_j$ after the first appearance of an excess edge $(v_i,v_j)$ in $W'$. The result is a kernel walk with $T$ as an induced tree.
\end{proof}

Tree walks of a given excess $\xi$ can be arbitrarily large.
However, our next result establishes
that there are only finitely many kernel walks of excess $\xi$.
This is reminiscent of the result of Wright~\cite{wright}
on the enumeration of connected graphs.

\begin{lemma}\label{lem:R_k}
Let $W_K$ be a kernel walk with excess $\xi$. Then its induced tree $T(W_K) = (V,E)$ satisfies $|V| \leq 3\xi-1$ and the number of its simple edges is at most $2\xi-2$. These bounds are tight.
Thus, $K_\xi(u,v,z)$ is a polynomial of degree $2\xi-2$ in $u$, $3\xi-1$ in $v$ and $4\xi-2$ in $z$.
\end{lemma}

\begin{proof}
Consider a kernel walk $W_K$ of excess $\xi$, with
$m$ vertices,
$\ell_1$ leaves,
$\ell_2$ vertices of degree $2$ that are not the root,
and outdegree sequence $(d_1, \ldots, d_m)$.
Each leaf is incident to an excess edge, so $\ell_1 \leq \xi$.
Each vertex of degree $2$ is incident to an excess edge, so
$\ell_2 \leq 2 \xi - \ell_1$.
The sum of the outdegrees is $m-1$, so
\[
    m - 1 = \sum_j d_j \geq 2\,(m-\ell_1-\ell_2) + \ell_2
\]
which implies $m \leq 3 \xi - 1$.
The number of simple edges is bounded by $m - 1 - \xi \leq 2 \xi - 2$.
The kernel walk has at most $2 \xi - 2$ half-steps along the simple edges,
and $2 \xi$ half-steps along the excess edges,
so the half-length is bounded by $4 \xi - 2$.

In order to prove that these bounds are tight, consider a binary tree on $2\xi-1$ vertices. It is well known that its number of leaves is $\xi$ to which we attach $\xi$ excess edges and define the tree walk $W$ by a modified depth-first-search where we consecutively traverse each edge incident to a leaf four times instead of only twice. It is easy to check that $W$ is indeed a kernel walk, since all leaves in $T(W)$ are incident to excess edges and every vertex which is not the root is either an internal vertex in the original binary tree and has therefore degree 3 or is incident to one of the attached excess edges. The number of simple edges in this tree is of course the number of edges in the binary tree, which is $2\xi-2$ and the number of vertices is $2\xi-1+\xi = 3\xi-1$.
\end{proof}

Although expressing $K_\xi(u,v,z)$ directly is challenging,
some subfamilies of kernel walks have a simple expression.
A kernel walk of excess $\xi$ is said to be \emph{optimal}
if it contains $2 \xi - 2$ simple edges, and \emph{near-optimal}
if it contains $2 \xi - 3$ simple edges.

\begin{lemma}\label{lem:optim}
Let $\cat(n)$ denote the $n$-th Catalan number.
There are $(3 \xi - 1)! \cat(\xi - 1)$
optimal kernels of excess $\xi$ for $\xi \geq 1$,
and $(3 \xi - 1)! \cat(\xi - 1)$ near-optimal kernels of excess $\xi$
for $\xi \geq 2$.
Let $K_{\xi,s}\left(z\right)$ denote the generating function of kernel walks
with excess $\xi$ and $s$ simple edges in the induced tree,
where $z$ marks the half-length of the walk.
This implies
    \begin{enumerate} 
    \item[(a)] $K_{\xi,2\xi-2}(z) = \cat(\xi-1)z^{4\xi-2}$, $\qquad \qquad \; \,$ for $\xi\geq 1$
    \item[(b)] $K_{\xi,2\xi-3}(z) = (3\xi-1)\cat(\xi-1)z^{4\xi-3}$, $\quad$ for $\xi\geq 2$.
    \end{enumerate}
\end{lemma}

\begin{proof}[Proof of Lemma~\ref{lem:optim}(a)]
Tree walks have labeled vertices, but the order of visit of the nodes by the walk induces a canonical embedding in the plane. We will soon prove that optimal (resp.~near-optimal) tree walks of excess $\xi$ have $4 \xi - 2$ (resp.~$4 \xi - 3$) nodes, so we can think without loss of generality about plane unlabeled trees, and decorate afterwards with one of the $(4 \xi - 2)!$ (resp.~$(4 \xi - 3)!$) possible labelings.

We already showed in the proof of Lemma~\ref{lem:R_k} that walks constructed from binary trees where excess edges are attached at their leaves are optimal. Since there are exactly $\cat(\xi-1)$ binary trees with $\xi$ leaves and the kernel walk is uniquely determined by the induced tree in this case, it is left to show that these walks are indeed the only trees induced by optimal kernel walks.

So let us assume for contradiction that there is an optimal walk $W$ which does not induce a binary tree with simple edges and $\xi$ leaves at which the excess edges are attached. Then (at least) one of three cases must occur:
    \begin{enumerate}
        \item there is a vertex $v$ with outdegree greater than $2$
        \item there is a vertex $v$ with outdegree $2$ which is incident to at least one outgoing excess edge
        \item there are less than $\xi$ leaves:
        \begin{enumerate}
            \item there is a vertex $v$ incident to an ingoing excess edge and one outgoing simple edge.
            \item there is a vertex $v$ incident to an ingoing excess edge and one outgoing excess edge or two outgoing simple edges
            \item there are less than $\xi$ excess edges in total.
        \end{enumerate}
    \end{enumerate}
    However, in each of these cases, we can modify the walk such that we increase the number of simple edges without increasing the excess edges in the induced tree by the following modifications which are also depicted in Figure~\ref{fig:inc_sim}. By Lemma \ref{a:prop:char} it is enough to explain these modifications on the tree as long as the excess edges do not decrease and the tree satisfies both conditions in the lemma. The modifications for the respective cases are as follows.
    \begin{enumerate} 
        \item Let $(v,u_1)$ and $(v,u_2)$ be two of at least three outgoing edges at $v$. Then we create a new vertex $w$, shift the edges $(v,u_1)$ and $(v,u_2)$ to $(w,u_1)$ and $(w,u_2)$ respectively and add one simple edge $(v,w)$. 
        \item Let $(v,u)$ an outgoing excess egde at $v$. We delete $(v,u)$, create a new vertex $w$ and insert a simple edge $(v,w)$ and an excess edge $(w,u)$.
        \item[(3a)] Let $(u_1,v)$ be the ingoing excess edge and $(v,u_2)$ the outgoing simple edge at $v$.  
        Now we delete the edge $(v,u_2)$ and add a simple edge $(u_1,u_2)$. Then we create a new vertex $w$, add an excess edge $(v,w),$ and color the edge $(u_1,v)$ as simple.
        \item[(3b)] Let $(u,v)$ be the excess edge incident to $v$ and $W_1$ be the subtree rooted at $v$. Note that since $v$ is incident to two outgoing simple edges or one outgoing excess edge, this subtree is the induced tree of a kernel walk by Lemma \ref{a:prop:char} and we can move it without the restriction that its root has to be incident to an excess edge. So we create two new vertices $w_1$ and $w_2$, color $(u,v)$ as a simple edge and add two new edges $(u,w_1)$, $(w_1,w_2)$ where $(u,w_1)$ is a simple edge and $(w_1,w_2)$ is an excess edge.
        \item[(3c)] In the case that there are less than $\xi$ excess edges in total, there exists an excess edge $(v,u)$ with excess $\xi(v,u) >1$. Now we create three new vertices $w_1, w_2, w_3$, color the edge $(v,u)$ simple, add a simple edge $(v,w_1)$ and distribute the excess as follows. We create an excess edge $(w_1,w_2)$ with excess $1$ and an excess edge $(u,w_3)$ with excess $\xi(v,u) -1$.
    \end{enumerate}
    In all cases, we added at least one simple edge and ended up with a tree where every leaf is incident to an excess edge and every inner vertex has outdegree at least two or is incident to an excess edge. That is, it is induced by a kernel walk by Lemma \ref{a:prop:char}.
    This contradicts the assumption that $W$ was an optimal walk and hence 
    \[
        K_{\xi,2\xi-2}(z) = \cat(\xi-1)z^{4\xi-2}.
    \]
\end{proof}
\begin{proof}[Proof of Lemma~\ref{lem:optim}(b)]
    Next, we turn our attention to near-optimal kernel walks. These are walks where we can increase the number of simple edges by exactly one if we apply the modifications in the previous proof. Thus, only the cases (1),(2) and (3a) are relevant and we count kernel walks where the induced tree is
    \begin{enumerate}
        \item \emph{a binary tree with simple edges and $\xi$ leaves where exactly one vertex has outdegree $3$ and the excess edges are attached at the leaves:}\\
        If there is a vertex with degree 3, we have two possibilities to extend it by an edge to a binary tree (see Figure \ref{fig:inc_sim}, Case 1) and the tree walk is still uniquely determined. Thus, equivalently we count half the number of optimal tree walks with one distinguished edge in the induced tree and obtain 
        \[
            \frac{\xi-2}{2}\cat(\xi-1)
        \]
        \item \emph{a binary tree with simple edges and $\xi$ leaves where the edge incident to one of the leaves is turned into an excess edge and the remaining excess edges are attached at the remaining $\xi-1$ leaves:}\\
        In this case we just count binary trees with $\xi$ leaves and choose one of the edges to be an excess edge. However, if this chosen edge is the left edge at a vertex, then we gain an additional possibility how the walk unfolds (\eg vuvwvuv or vuvuvwv). Note that we count the same tree reflected at the root as well where all exactly those edges which do not contribute in the original tree contribute an additional possibility. Thus, by symmetry we count additionally $\xi/2$ times binary trees with $\xi$ leaves and obtain in total 
        \[
            \left(\xi+\frac{\xi}{2}\right)\cat(\xi-1)
        \]
        \item \emph{a binary tree with simple edges and $\xi-1$ leaves at which $\xi-1$ excess edges are attached and where we insert an ingoing sequence of an excess edges followed by a simple edge at an inner vertex:}\\
        There are $2\xi-3$ vertices in a binary tree with $\xi-1$ vertices where we could insert this incoming sequence before the original incoming simple edge. Further this sequence gives two distinct ways how the walk unfolds ($vwvwW_1wv$ or $vwW_1wvwv$, where $W_1$ is a subwalk starting at $w$). Thus, using a simple identity of the Catalan numbers, the number of walks we count is 
        \[
            2(2\xi-3)\cat(\xi-2) = \xi\cat(\xi-1).
        \]
    \end{enumerate}
    Consequently, the number of near-optimal walks is in total 
    \[
        \left(\frac{\xi-2}{2}+\frac{3\xi}{2}+\xi\right)\cat(\xi-1)=\left(3\xi-1\right)\cat(\xi-1)
    \]
    and they are all of length $8\xi-6$.
\end{proof}

\begin{figure}
    \begin{subfigure}{0.3\textwidth}
        \centering
        \includegraphics[width=\textwidth, page=2]{walk.pdf}
        \caption{Case 1}
    \end{subfigure}
    \hspace{8mm}
    \begin{subfigure}{0.3\textwidth}
        \centering
        \includegraphics[width=\textwidth, page=1]{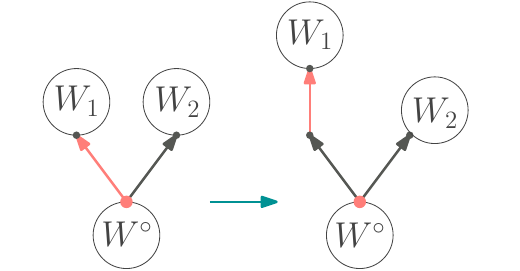}
        \caption{Case 2}
    \end{subfigure}
    \\
    \begin{subfigure}{0.3\textwidth}
        \centering
        \includegraphics[width=\textwidth, page=3]{walk.pdf}
        \caption{Case 3a}
    \end{subfigure}
    \hspace{4mm}
    \begin{subfigure}{0.3\textwidth}
        \centering
        \includegraphics[width=\textwidth, page=5]{walk.pdf}
        \caption{Case 3b}
    \end{subfigure}
    \hspace{4mm}
    \begin{subfigure}{0.3\textwidth}
        \centering
        \includegraphics[width=\textwidth, page=4]{walk.pdf}
        \caption{Case 3c}
    \end{subfigure}
    \caption{Increasing the number of simple edges}
    \label{fig:inc_sim}
\end{figure}

Given a kernel walk $W_K$ with excess $\xi$, we reconstruct a tree walk $W$ by substituting every simple edge by a sequence of back and forth steps, adding a sequence of steps at the root of $T(W_K)$, moving the root to the leaf of this attached path and adding a tree walk without excess at the beginning and after each step in this extension of $W_K$.

\begin{lemma}\label{lem:T_k}
    Let $W_\xi(z)$ be the generating function of the number of tree walks with excess $\xi \geq 1$ and $K_\xi(u,v,z)$ the generating function of kernel walks with excess $\xi$ and where $u$ marks the number of simple edges, $v$ the number of vertices and $z$ the half-length of the walk. Then
    \[
        W_\xi(z) = \frac{C(z)}{1-zC(z)^2} K_\xi\left(\frac{1}{1-zC(z)^2},1,zC(z)^2\right).
    \]
\end{lemma}

\begin{proof}
    Given a kernel walk $W_K$ with excess $\xi$, we want to reconstruct all possible tree walks $W$ which reduce to $W_K$. Therefore we reverse the procedure in Definition \ref{def:kernels} in every possible way. If we undo an iteration of Step~4, we add a vertex in $T(W')$ on a simple edge and extend $W'$ accordingly. Thus, going back all iterations of Step~4, we substitute every simple edge by a (possibly empty) sequence of vertices which add a back and forth step to the length. In terms of the generating function we therefore obtain
    \[
        K_\xi\left(\frac{1}{1-z},1,z\right),
    \]
    where we set $v=1$ since we do not need the number of vertices as a parameter.
    Reversing Step 3, breaks down to adding a sequence of back and forth steps at the root of $T(W')$ and moving the root to the leaf of this attached path. Naturally, this means we add a factor $\frac{1}{1-z}$ to the generating function and obtain
    \[
        \frac{1}{1-z} K_\xi\left(\frac{1}{1-z},1,z\right).
    \]
    Finally, we undo all iterations of Step~2, by adding a tree walk without excess at the beginning and after each step in $W'$. The corresponding generating function for tree walks without excess is well known to be the generating function of contour walks of trees~\cite{Bushygraphs, inna_z, physics}. That is, we multiply two factors of $W_0(z) = C(z)$ with each $z$ since it counts a back and forth step and we add an additional factor of $C(z)$ to the equation, such that in total we obtain
    \[
        \frac{C(z)}{1-zC(z)^2} K_\xi\left(\frac{1}{1-zC(z)^2},1,zC(z)^2\right).
    \]
\end{proof}

\begin{proof}[Proof (Theorem \ref{thm:strucT_k})]
    By the previous lemma, we have
    \[
        W_\xi(z) = \frac{C(z)}{1-zC(z)^2} K_\xi\left(\frac{1}{1-zC(z)^2},1,zC(z)^2\right)
    \]
    and by Lemma \ref{lem:R_k} we know that $K_
    \xi(u,v,z)$ is a polynomial with degree at most $4\xi-2$ in $z$ and $2\xi-2$ in $u$. So if we expand $K_\xi(u,1,z)$ in $u$, we obtain the desired form.
    Further, consider a kernel walk with $s$ simple edges,
    excess $\xi$, length $2\ell$ and $m$ edges.
    Then the half-length $\ell = \xi + m$, and the number of edges
    is bounded by $\xi + s$, so
    $\ell \leq 2\xi + s$. This implies that
    the degree of $[u^s]K_\xi(u,1,z)$ in $z$ is at most $2\xi+s$. Again starting from a binary tree with $s$ edges if $s$ is even and all excess edges attached to the leaves, the maximum degree $2\xi+s$ of $[u^s]W_\xi(u,1,z)$ is attained. If $s$ is odd, we take a binary tree on $s-1$ edges attach another simple edge at the neighbour of one of the leaves and all excess edges attached at the newly formed leaves and we attain again the maximum degree.\\
    Finally, we computed $K_{\xi,2\xi-2}(x)$ and $K_{\xi,2\xi-3}(x)$ in Lemma~\ref{lem:optim}. 
\end{proof}

    \subsection{A recursion for the generating function of tree walks of excess $\xi$}
    \label{sec:recursion:kernels}

Theorem~\ref{thm:strucT_k} raises the question of the computation of the generating function $K_{\xi,s}(z)$ of kernel walks of excess $\xi$ with $s$ simple edges, where $z$ marks the half-length.
Technically, we can deduce the following functional equation for the generating function of kernel walks with one additional auxiliary parameter.

\begin{proposition}\label{prop:ker}
Let $K(x,u,v,z)$ be the generating function of kernel walks. That is,
\[
    K(x,u,v,z) = \sum_{j,s,\ell \geq 0} k_{j,s,m,2\ell} \, \frac{x^j}{j!}u^s \frac{v^{m}}{m!}z^{\ell}
\]
where $u$ counts the number of simple edges in the induced tree, $v$ the number of vertices, $z$ the half-length of the walk and $x$ how often the walk leaves the root. Then
\begin{align*}
    K(x,u,v,z) &= v\exp \Bigg( \laplace_{t=1} \Big(xuz(K(t,u,v,z)-v)+D\left(t,xz\right)K(t,u,v,z)\Big)\Bigg)\\
    &\quad -v\laplace_{x=1}\Big(xuz(K(x,u,v,z)-v)\Big)
\end{align*}
where $D(t,z) = \sum_{j \geq 2} \frac{z^jt^{j-1}}{j! (j-1)!}$ and $\laplace_{t=1} \left( A(t) \right) = \sum_{i\geq 0} i!\, [t^i] A(t)$.
\end{proposition}

However, the operator $\laplace_{t=1}$ and the modified Bessel function $D(t,z)$ make the analysis of the equation quite complicated. We omit the proof as it is analogous to the proof of Lemma~\ref{lem:super} below and instead boil down the recursion even further to superreduced walks. In this case the complexity of the problem is fully isolated in the equation for the generating function of superreduced walks and its computational evaluation of the coefficients for small excess is significantly faster.

\begin{definition}
A superreduced walk is a tree walk where no edge is simple.
Denoting by $s_{m,2\ell}$ the number of such walks
of length $2\ell$ on $m$ vertices,
their generating function is
\[
    S(v,z) =
    \sum_{\substack{\ell \geq 0\\ m\geq 0}}
    s_{m,2\ell}
    \frac{v^m}{m!}
    z^{\ell}.
\]
\end{definition}

The parameters of superreduced walks are bounded as follows.

\begin{proposition}
    A superreduced walk with excess $\xi$ is of length at most $4\xi$ and it spans a tree with at most $\xi+1$ vertices. Their generating function $S_\xi(v,z)$ is therefore a polynomial of degree $\xi+1$ in $v$ and $2\xi$ in $z$.
\end{proposition}

\begin{proof}
    Each of the edges in the spanned tree of a superreduced walk has to be an excess edge and the tree thus has at most $\xi+1$ vertices. Further, each edge is crossed twice plus the excess steps. That is the length of the walk is $2|E|+2\xi \leq 4\xi$ steps. 
\end{proof}

The following lemma reduces the enumeration of kernels
to the enumeration of superreduced walks.
The main idea is to consider the induced tree of a kernel walk and isolate the component which contains the root after deleting all simple edges (see Figure \ref{fig:decomp_kernel}). The restriction of the kernel walk to this component is a superreduced walk and the restriction of the kernel walk to all of the other components are kernel walks again.  

\begin{lemma} \label{th:link:kernel:superreduced}
    Let $S(v,z)$ be the generating function of superreduced walks, that is, kernel walks without simple edges, where $v$ counts the number of vertices in the induced tree and $z$ the half-length of the walk. Then for the generating function of kernel walks
    $K(u,v,z) = \sum_{\xi \geq 0} K_\xi(u,v,z)$
    it holds that
    \[
        K(u,v,z) = \frac{1}{\big(1-uz(K(u,v,z)-v)\big)}S\left(v, \frac{z}{\big(1-uz(K(u,v,z)-v)\big)^2}\right)-uvz(K(u,v,z)-v).
    \]
\end{lemma}

\begin{proof}
    Let $\mF$ denote the family of tree walks
    where the root is attached to a single simple edge
    which links it to a kernel walk on at least two vertices.
    Using $u$ to mark simple edges,
    $v$ to mark vertices and $z$ to mark the half-length of the walk,
    the generating function of $\mF$ is
    \[
        F(u,v,z) = u v z (K(u,v,z) - v).
    \]
    Let $\mG$ denote the family of tree walks
    obtained from a superreduced walk by inserting
    a sequence of walks from $\mF$
    after each step and an additional sequence before the first step.
    Then the generating function of $\mG$ is
    \[
        G(u,v,z) =
        \frac{1}{1 - u z (K(u,v,z) - v)}
        S \left(
            v,
            \frac{z}{(1 - u z (K(u,v,z) - v))^2}
        \right).
    \]
    Consider a kernel walk $K$.
    Let $S$ denote the component containing the root
    when all simple edges of $K$ are removed.
    By definition, $S$ is a superreduced walk.
    This decomposition shows that each kernel
    corresponds to a unique element of $\mG$.
    Reciprocally, all elements from $\mG$ are kernels,
    except the walks obtained by starting with $S$
    reduced to a single vertex, to which
    a unique element of $\mF$ is attached.
    Those walks have generating function
    \[
        u v z(K(u,v,z) - v).
    \]
    Indeed, those walks are not kernels,
    since the root of the tree would be moved in Step 3
    of the procedure in Definition \ref{def:kernels}.
    Thus, the generating function of kernels is
    \[
        \frac{1}{1 - u z (K(u,v,z) - v)}
        S \left(
            v,
            \frac{z}{\left(1 - u z^2 (K(u,v,z) - v)\right)^2}
        \right)
        - u v z(K(u,v,z) - v).
    \]
\end{proof}

\begin{figure}
    \centering
        \includegraphics[width=0.25\textwidth, page = 10]{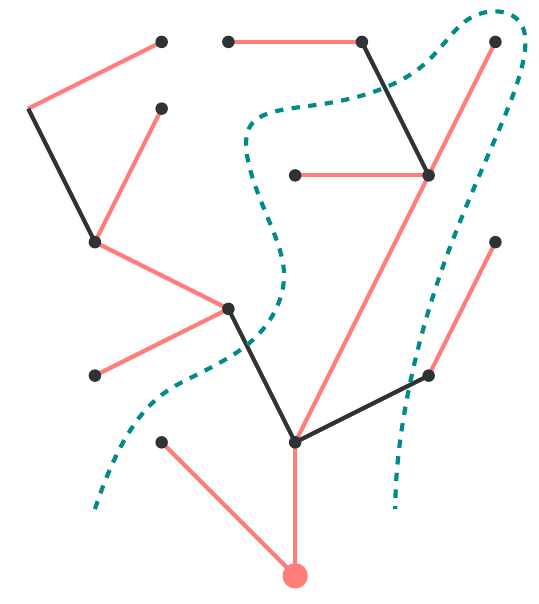} 
         \includegraphics[width=0.25\textwidth, page = 11]{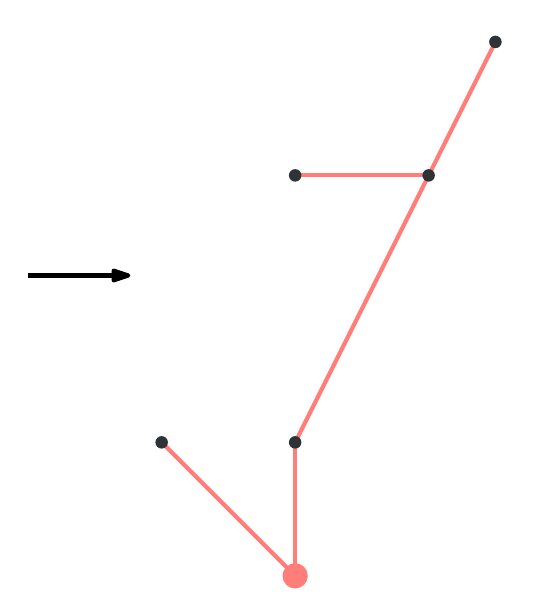}
        \caption{Decomposing kernel walks by isolating the superreduced component including the root (red)}
        \label{fig:decomp_kernel}
\end{figure}

Once given the generating function $S(v,z)$ of superreduced walks, we compute $K(u,v,z)$ by Lagrange inversion (see \eg \cite{gessel2016lagrange}).

\begin{lemma}\label{a:lem:lagrange}
Let $S(v,z)$ denote the generating function
of superreduced walks,
and $K(u,v,z)$ the generating function of kernels,
where $u$ marks the simple edges, $v$ the vertices
and $z$ the half-length. Then
\[
    K(u,v,z) =
    S(v,z)
    + \sum_{s \geq 1}
    \frac{u^s}{s+1}[t^s]\left(\frac{1}{1-zt}S\left(v,\frac{z}{(1-zt)^2}\right)-v z t - v\right)^{s+1}.
\]
Let $K_{\xi}(u,v,z)$, $S_{\xi}(v,z)$ and $S_{\leq \xi}(v,z) = \sum_{k=0}^{\xi} S_k(v,z)$
denote the generating functions of kernels of excess $\xi$, superreduced walks of excess $\xi$
and superreduced walks of excess at most $\xi$, then
\[
    K_{\xi}(u,v,z) =
    S_{\xi}(v,z)
    + \sum_{s = 1}^{2 \xi - 2}
    \frac{u^s}{s+1}[t^s y^{\xi-1}]\left(\frac{1}{1 - y z t} S_{\leq \xi}\left(\frac{v}{y},\frac{y z}{(1-y z t)^2}\right) - v z t - \frac{v}{y}\right)^{s+1}.
\]
\end{lemma}

\begin{proof}
Set $L(u,v,z) = u(K(u,v,z) - v)$ and slightly reformulate the equation of the previous lemma to
\begin{align*}
    L(u,v,z) = u\left(\frac{1}{1-z L(u,v,z)}S\left(v,\frac{z}{(1-zL(u,v,z))^2}\right)-v z L(u,v,z) - v\right)
\end{align*}
which satisfies the condition
\[
    L(u,v,z) = u \cdot \Phi(L(u,v,z),v,z)
\]
for
\[
    \Phi(t,v,z) = \frac{1}{1-zt}S\left(v,\frac{z}{(1-zt)^2}\right)-vz t - v.
\] 
Applying the Lagrange inversion theorem, we finally obtain
\[
    [u^s] u(K(u,v,z)-v) = [u^s] L(u,v,z) = \frac{1}{s}[t^{s-1}]\left(\frac{1}{1-zt}S\left(v,\frac{z}{(1-zt)^2}\right)-v z t - v\right)^s.
\]
By definition, a kernel without simple edge is superreduced, so $[u^0] K(u,v,z) = S(v,z)$.
For the last equation, we deduce for any $s \geq 1$
\[
    [u^s] K(u,v,z) =
    \frac{1}{s+1}[t^s]\left(\frac{1}{1-zt}S\left(v,\frac{z}{(1-zt)^2}\right)-v z t - v\right)^{s+1}.
\]
The first result of the lemma follows
by multiplying by $u^s$ and summing over $s$.

For the second result,
we use the fact that kernels of excess $\xi$ contain at most $2 \xi - 2$ simple edges,
so computing $K_{\xi}(u,v,z)$ requires only
summing up to $s = 2 \xi - 2$.
Further, since the excess is the half-length minus the number of vertices plus one,
we have
\[
    K_{\xi}(u,v,z) =
    [y^{\xi-1}]
    K \left( u, \frac{v}{y}, y z \right),
    \qquad
    S_{\xi}(u,v,z) =
    [y^{\xi-1}]
    S \left( u, \frac{v}{y}, y z \right),
\]
which provides the second result.
\end{proof}

Our next lemma provides an equation characterizing $S(v,z)$,
from which $S_{\leq \xi}(v,z)$ is computable.
The proof relies on the idea from \cite{physics}
to mark the number of times the walk leaves the root (see Figure \ref{fig:decomp_tree}). Applying the symbolic method \cite{BLL97, FS09}
to translate it into generating functions results in a series $S(x,v,z)$ for superreduced walks,
where a new auxiliary variable $x$
marks how often the walk leaves the root.

\begin{lemma}\label{lem:super}
    Let $s_{j,m,2\ell}$ denote the number of superreduced walks
    on $m$ vertices, length $2\ell$
    and leaving the root $j$ times.
    Let
    \[
        S(x,v,z) =
        \sum_{j,m,\ell \geq 0}
        s_{j,m,2\ell}
        \frac{x^j}{j!}
        \frac{v^m}{m!}
        z^{\ell}
    \]
    denote the generating function of superreduced kernel walks, where $z$ marks the half-length of the walk, $v$ the number of vertices in the induced tree and $x$ how often the walk leaves the root. Then
    \[
        S(x,v,z) = v \exp \left( \laplace_{t=1} \Big(D\left(t,xz\right)S(t,v,z)\right)\Big)
    \]
    where $D(t,x) = \sum_{k \geq 1} \frac{x^{k+1}}{(k+1)!} \frac{t^k}{k!}$ and $\laplace_{t=1} \left( A(t) \right) = \sum_{k \geq 0} k!\, [t^k] A(t)$.
\end{lemma}

\begin{proof}
Let $\mF$ denote the family of superreduced walks
composed of a root of degree $1$,
linked by an excess edge
to a superreduced walk.
Let $f_{k,j,m,2\ell}$ denote the number of such walks
of length $2 \ell$, on $m$ vertices,
leaving the root $j$ times,
and leaving the only child of the root $k+1$ times.
We associate to the family $\mF$ the generating function
\[
    F(t,x,v,z) =
    \sum_{k,j,m,\ell \geq 0}
    f_{k,j,m,2\ell}
    \frac{t^k}{k!}
    \frac{x^j}{j!}
    \frac{v^m}{m!}
    z^{\ell}.
\]
Consider a walk $W \in \mF$,
let $r$ denote its root,
$c$ the only child of $r$,
and $W'$ the walk of root $c$
obtained from $W$ by removing $r$.
Assume the number of steps leaving $c$ is $k+1$.
We construct two disjoint sets $A$ and $B$
whose union is $\{1, 2, \ldots, k\}$.
For all $n \in \{1,2,\ldots,k\}$,
if the $n$th step leaving $c$
goes to the root $r$, then $n \in A$.
Otherwise, $n \in B$.
We did not include the $(k+1)$th step leaving $c$,
as we know it must always reach $r$.
Observe that knowing $A$, $B$ and $W'$
is enough to reconstruct $W$.
Further, if $A$ has size $h$,
then $W$ contains $h+1$ steps leaving its root $r$
and the length of $W$ is the length of $W'$ plus $2 (h+1)$.
Because of the choice to make $t$ an exponential variable
in the generating function $F(t,x,v,z)$,
this bijection implies
\[
    F(t,x,v,z) =
    D(t, x z) S(t,v,z).
\]
In order to forget the variable $t$
in the generating function $F(t,x,v,z)$,
we use the Laplace operator
\[
    \sum_{k,j,m,\ell \geq 0}
    f_{k,j,i,2\ell}
    \frac{x^j}{j!}
    \frac{v^m}{m!}
    z^{\ell}
    =
    \laplace_{t=1}(F,t,x,v,z).
\]
Finally, a superreduced tree walk is a root
and a set (possibly empty)
of elements from $\mF$, so
\[
    S(x,v,z) = v \exp \left( \laplace_{t=1}(F,t,x,v,z) \right).
\]
\end{proof}

\begin{figure}
        \centering
        \includegraphics[width=0.3\textwidth, page=7]{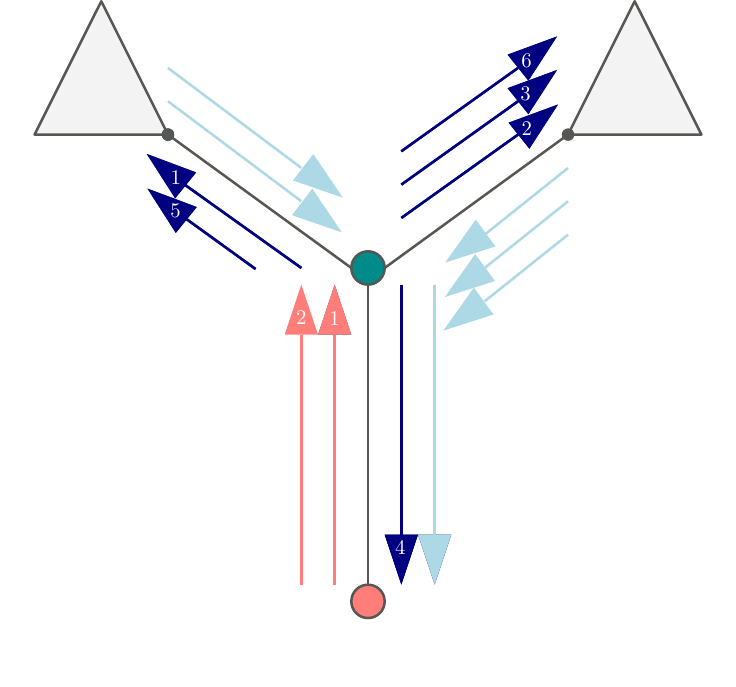}
        \caption{Decomposition of a superreduced walk}
        \label{fig:decomp_tree}
\end{figure}

By implementing this well founded recursion in $v$ \emph{and} $z$ it is easy to compute $W_\xi(z)$ for small excess $\xi$. We only need to compute superreduced walks with excess at most $\xi$. That is, we compute the coefficients of $S(v,z)$ up to order $\xi+1$ in $v$ and $2\xi$ in $z$.
Then we compute the first terms of $K(u,v,z)$ which are $K_1(v,z), K_2(v,z), \dots, $ $K_\xi(v,z)$ by Lemma~\ref{a:lem:lagrange}.
The generating functions which we obtained for superreduced walks with fixed excess $\xi$, denoted by $S_\xi(z) = [y^{\xi}] yS\left(\frac{1}{y},\sqrt{y}z\right)$ and  $K_\xi(u,1,z)$, where $\xi \leq 5$ are given in Table \ref{tab:my_label}. The size of the polynomials illustrates well why the reduction to superreduced walks is computationally more powerful.

\begin{table}
    \centering
    \begin{tabular}{c|c|c}
         $\xi$ & $S_{\xi}(z)$ & $K_\xi(u,1,z)$
         \\ &&\\
         \hline &&\\
         0 & $1$ & $1$ \\ &&\\
         \hline &&\\
         1 & $z^2$ & $z^2$\\ &&\\
         \hline &&\\
         2 & $6z^4 + z^3$ & $u^2z^{6} + 5uz^{5} + 6z^4 + z^3$ \\ &&\\
         \hline &&\\
         3 & $57z^{6} + 20z^{5} + z^{4}$ & $2u^4z^{10} + 16u^3z^{9} + 52u^2z^{8} + 2u^2z^{7}$  \\
          & & $+ 84uz^{7} + 12uz^{6} + 57z^{6} + 20z^{5} + z^4$  \\ &&\\
         \hline &&\\
         4 & $678z^{8} + 378z^{7} + 50z^{6} + z^{5}$ & $5u^6z^{14} + 55u^5z^{13} + 267u^4z^{12} + 6u^4z^{11}+ 745u^3z^{11}$  \\
          & & $  + 54u^3z^{10} + 1290u^2z^{10} + 205u^2z^{9} + 3u^2z^{8} +1350uz^{9}$  \\
         
         & & $ + 416uz^{8} + 21uz^{7} + 678z^{8} + 378z^{7} + 50z^{6} + z^{5}$  \\ &&\\
         \hline &&\\
         5 & $9270z^{10} + 7272z^{9} + 1684z^{8}$ & $14u^8z^{18} + 196u^7z^{17} + 1254u^6z^{16} + 20u^6z^{15}$  \\ 
          & $+ 112z^{7} + z^{6}$ & $ + 4836u^5z^{15} + 240u^5z^{14} + 12453u^4z^{14} + 1296u^4z^{13} $
          \\ 
          &  & $ + 12u^4z^{12} + 8(2787u^3z^{13} + 517u^3z^{12} + 15u^3z^{11})$ \\
          & & $+ 2(13854u^2z^{12} + 4251u^2z^{11} + 266u^2z^{10} + 2u^2z^{9})$
          \\ 
          & & $+ 4(5610uz^{11} + 2775uz^{10} + 342uz^{9} + 8uz^{8})$\\ 
          & & $+ 9270z^{10} + 7272z^{9} + 1684z^{8} + 112z^{7} + z^{6}$
    \end{tabular}
    \caption{Caption}
    \label{tab:my_label}
\end{table}

For example for $\xi=1,2,3$, we obtain the generating functions
\begin{align*}
    W_1(z) &= \frac{z^2C\left(z\right)^5}{1-zC\left(z\right)^2}\\
    W_2(z) &= C\left(z\right)\left(\frac{z^3C\left(z\right)^{6}+4z^4C\left(z\right)^8-6z^{5}C\left(z\right)^{10}+2z^{6}C\left(z\right)^{12}}{\left(1-zC\left(z\right)^2\right)^3}\right)\\
    W_3(z) &= z^4C\left(z\right)^{9}
    \left(\frac{1+ 16zC\left(z\right)^2+11z^{6}C\left(z\right)^{12}  + 95z^4C\left(z\right)^8}{\left(1-zC\left(z\right)^2\right)^5}\right)\\
    &\quad - z^4C\left(z\right)^{9}\left(\frac{54z^{5}C\left(z\right)^{10}+62z^3C\left(z\right)^6 + 5z^2C\left(z\right)^4 }{\left(1-zC\left(z\right)^2\right)^5}\right),
\end{align*}
recovering and extending
the results of \cite{physics}
and of \cite{Bushygraphs}
(except for $W_2(z)$ where our calculation differ
from \cite{Bushygraphs} and agree with \cite{physics}). It is further straight forward to confirm the conjectured closed expressions of the coefficients in \cite[p~31]{physics}.
While the authors also gave asymptotics of $w_{m,\ell}$ for fixed $m$ as $\ell$ grows larger in terms of the Stirling numbers of the second kind, our decomposition provides asymptotics for arbitrary fixed excess and $\ell$ tending to infinity. A deeper understanding of the polynomials $K_\xi(u,v,z)$ in Lemma \ref{lem:R_k} would allow us to fully describe the (bivariate) asymptotics for walks of length $2(m+\xi-1)$ and excess $\xi$ and therefore the full asymptotics of the moments of the spectral measure. The presence of a Laplace transform in the expression of $S(x,u,v,z)$ will be an interesting challenge for this analysis.

        \section{A refined normalisation of the spectral measure and some curious identities}\label{sec:cons}

In this section, we return to our initial motivation to describe the moments of the spectral measure $\mu^c$ by the identity
\[
    M_{\mu^c}(z) = \frac{1}{c} W\left(c,\frac{z^2}{c}\right) = \sum_{\xi\geq 0} \frac{1}{c^\xi} W_\xi\left(z^2\right).
\]
As Zacharevich~\cite{inna_z} pointed out, $\mu^c$ is fully determined by its moments and if $\mu^c$ were a continuous measure, we could compute its density by the inversion formula of Stieltjes-Perron. This is not the case ($\mu^c$ has a dense set of atoms \cite{chayes1986density,spec3}), but nonetheless a better understanding of the Stieltjes transform of $\mu^c$ would entail a better understanding of the measure itself.

In combinatorial terms, the Stieltjes transform $S_\mu(z)$ of a measure $\mu$ with finite moments is simply the ordinary generating function of moments evaluated at $z^{-1}$ multiplied by $z^{-1}$. That is,
\[
    S_{\mu}(z) = \sum_{\ell \geq 0} m_{\ell}(\mu) z^{-(\ell+1)}.
\]
In turn, under some conditions, the Stieltjes-Perron formula
expresses the density $\rho$ of the measure $\mu$ by
\begin{equation}\label{eq:inversion}
    \rho(z) =
    \lim_{\varepsilon \rightarrow 0} -\frac{1}{\pi} \mbox{Im}\left(S_\mu(z+i\varepsilon)\right).
\end{equation}
For example, the Stieltjes transform of the limit law $\mu$ of the normalized spectral measure of $G(n,p)$ with $p$ constant, and its density,
are respectively
\[
    S_{\mu}(z) = \frac{1}{z}C\left(\frac{1}{z^{2}}\right),
    \qquad
    \lim_{\varepsilon \rightarrow 0} -\frac{1}{\pi} \mbox{Im}\left(S_\mu(z+i\varepsilon)\right) = \frac{\sqrt{4-z^2}}{2\pi} \mathds{1}_{(-2,2)}(z).
\]
The distribution given by this density is called after its shape, the \emph{semicircle distribution}.
The Stieltjes transform of $\mu^c$ equals
\[
    S_{\mu^c}(z) = \frac{1}{z} M_{\mu^c}\left(\frac{1}{z}\right)= \sum_{\xi \geq 0} \frac{1}{zc^\xi} W_\xi \left(\frac{1}{z^2}\right).
\]
Given the structure of $W_{\xi}(z)$ from Theorem~\ref{thm:strucT_k},
$S_{\mu^c}(z)$ is a sum of rational functions in $S_\mu(z)$
\[
    S_{\mu^c}(z) = S_\mu(z)+\frac{1}{c}\cdot\frac{S_\mu(z)^5}{1-S_\mu(z)^2} + \frac{1}{c^2}\cdot\frac{S_\mu(z)^7+4S_\mu(z)^9-6S_\mu(z)^{11}+2S_\mu(z)^{13}}{\left(1-S_\mu(z)^2\right)^3} +\dots
\]
Now one could hope that the inversion formula applied to each of the $z^{-1}W_\xi(z^{-1})$ would yield a density of a measure and the density of~$\mu_c$ would turn out to be a weighted sum of them. This hope is certainly too far fetched, as $\mu^c$ has a dense set of atoms. But Enriquez and Ménard \cite{Bushygraphs} found a way to still make use of this expansion by using a dilation operator in their Theorem 3.
The main idea is to scale the spectral measure and evaluate $M_{\mu^c}(z)$ at $z/(1+\frac{1}{2c})$ instead. This scaling entails a perturbation on the level of coefficients of $1/c$. In particular,
$\sum_{\ell \geq 0}
    m_{2\ell}
    \left( \frac{z}{1+\frac{1}{2c}}\right)^{2\ell}$
is equal to
\[
    \sum_{\ell\geq 0}
    \textstyle{
    \left(
    w_{0,2\ell}z^{2\ell}
    + \frac{1}{c}(w_{1,2\ell}-\ell w_{0,2\ell})z^{2\ell}
    +\frac{1}{c^2} \left(w_{2,2\ell} -\ell w_{1,2\ell} + \left(\frac{\ell^2}{2}+\frac{\ell}{4}\right)w_{0,2\ell}\right)z^{2\ell} + \dots
    \right).}
\]
Now the generating functions at $c^{-1}$ and $c^{-2}$
are polynomials in $z^2 C(z^2)^2$ multiplied by $C(z^2)$,
and the corresponding densities can be computed with the inversion formula.
The result of Enriquez and Ménard that we extend is

\begin{theorem}[Enriquez, Ménard~\cite{Bushygraphs}] 
    Let $\mu^c$ be the limit of the normalised spectral measure, let $\sigma$ be the semi-circle law and let $\sigma^{\{1\}}$ and $\sigma^{\{2\}}$ be the signed measures with null mass and densities\footnote{Note that $f^{\{1\}}(z)$ is the same as in~\cite{Bushygraphs} after factoring the nominator.}
    \[
        f^{\{1\}}(z) = \frac{1}{2\pi}(1-z^2)\sqrt{4-z^2}\mathds{1}_{(-2,2)}(z), 
    \]
    and
    \[
        f^{\{2\}}(z) = -\frac{1}{\pi}\left(\frac{21}{4}-\frac{325}{8}z^2+46z^4-17z^6+2z^8\right)\frac{1}{\sqrt{4-z^2}}\mathds{1}_{(-2,2)}(z)
        \]
    respectively.
    The moments of $\mu^c$ satisfy the following asymptotic expansion
    \[
        m_\ell\left(\mu^c\right) = m_\ell\left(\Lambda_{1+\frac{1}{2c}}\left(\sigma + \frac{1}{c}\sigma^{\{1\}}+\frac{1}{c^2}\sigma^{\{2\}}\right)\right)+o\left(\frac{1}{c^2}\right),
    \]
    where $\Lambda_{\alpha}$ denote the dilation operator that transforms a measure $\mu$ for $\alpha>0$ into the measure $\Lambda_\alpha(\mu)$ satisfying for every Borel set $A$ that $\Lambda_\alpha(\mu)(A) = \mu(A/\alpha)$.
\end{theorem}

We expand their calculation to order $5$ instead of $2$.
Instead of using the dilation operator, we can rescale the adjacency matrix $A(G(n,c/n))$ of $G(n,c/n)$ by $\frac{1}{\sqrt{c\, p(1/c)}}$ instead of $\frac{1}{\sqrt{c}}$. We define
\[
    \mu_n^{p} =
    \frac{1}{n}
    \sum_{\lambda \in Sp\left((c\, p(1/c))^{-1/2}A(G(n,c/n))\right)}
    \delta_\lambda,
\]
where $p(x)$ is a polynomial in $x$ with constant term $1$ which is yet to be determined, and $\mu^p$ for the limit as $n$ tends to infinity.
This implies
\[
    \mu^p_n = \Lambda_{p(1/c)^{-1/2}}(\mu^c_n)
    \quad \text{and} \quad
    M_{\mu^p}(z) = M_{\mu^c}\left(\frac{z}{\sqrt{p(1/c)}}\right).
\]
Hence, if $p(x) = 1+x+\frac{1}{4}x^2$, we end up with the same scaling as Enriquez and Ménard used for their expansion. However, we can choose $p(x)$ such that $[c^{-i}] M_{{\mu}^p}\left(z\right)$ is a polynomial in $\left(zC(z^2)\right)^2$ multiplied by $C(z^2)$ for all $0\leq i \leq 5$.

The original scaling factor $1/\sqrt{c}$ derives from the classical scaling of Wigner matrices, where one scales the matrix by $1/\sqrt{n\mathbb{V}(X)}$, where $X$ is distributed as the individual  matrix entries. In the case of adjacency matrices of $G(n,c/n)$ the variance of Bernoulli variables determining the entries of the matrix is of course $c/n(1-c/n)$ such that we obtain the scaling factor $1/\sqrt{c}$ in the limit.
We do not have a similar interpretation for our proposed alternative scaling.

\begin{proposition}\label{cor:Vi}
    Let $p_5(x) = 1+x+x^2+4x^3+33x^4+386x^5$ and $M_{\mu^c}(z)$ be the ordinary moment generating function of $\mu^c$ as defined in~(\ref{eq:ordGF}). Then for $V_i(z) = \left[c^{-i}\right] M_{\mu^p}\left(\sqrt{z}\right)$ we have
    \[
        V_i(z) = C(z)Q_i\left(zC(z)^2\right), \quad i=0,1,2,\dots,5,
    \]
    where
    \begin{align*}
        Q_0(x) &= 1,\qquad 
        Q_1(x) = -x, \qquad 
        Q_2(x) = -2x^3,\\
        Q_3(x) &= -\left(11x^{5} + x^4 - 2x^3 + 2x^2 + 3x\right), \\
        Q_4(x) &= -\left(90x^{7} + 27x^{6} - 19x^{5} + 17x^4 + 23x^3
        + 20x^2 + 26x\right),\\
        Q_5(x) &= -\left(931x^{9} + 529x^{8} - 163x^{7} + 166x^{6}  + 301x^{5} + 239x^4 + 249x^3 + 266x^2 + 324x\right).
    \end{align*}
\end{proposition}

The computation of $p_5(x)$ is straight forward and can be extended to an algorithm that computes a sequence of coefficients $(a_n)_{n\geq 0}$ with $a_0=1$ such that a scaling by the polynomial $\sum_{i=0}^N a_i x^i$ satisfies Conjecture~\ref{conj} up to $i=N$ (if it is true). 
In particular, for a polynomial $p_N(x) = \sum_{i=0}^N a_i x^i$, we have
    \[
        M_{\mu^p}\left(\sqrt{z}\right) = \sum_{\ell \geq 0} m_{2\ell} \frac{z^{\ell}}{p_N(1/c)^\ell} = \sum_{\xi\geq 0} \frac{1}{c^\xi} W_\xi\left( \frac{z}{p_N(1/c)} \right).
    \]
    The coefficient at $c^{-i}$ is therefore given by derivatives of the functions $(W_\xi(z))_{0\leq i\leq \xi}$. To simplify the computation, we define $q_N(x) = \sum_{i\geq 0} b_ix^i = p_N(x)^{-1}$ to be the inverse of $p_N(x)$ in the ring of formal power series. Then we have by Faà di Bruno's formula,
    \begin{align}
        \left[c^{-i}\right]M_{\mu^p}\left(\sqrt{z}\right) &=  \sum_{\xi = 0}^{i}\frac{1}{(i-\xi)!}\frac{\partial^{i-\xi} }{\partial x^{i-\xi}} W_\xi\big(zq_N(x)\big)\bigg|_{x=0}\nonumber \\
        &= W_{i}(z)+\sum_{\xi = 0}^{i-1} \frac{1}{(i-\xi)!} \sum_{k=1}^{i-\xi}z^kW_\xi^{(k)}\left(z\right)\hat{B}_{i-\xi,k}\left(b_1,b_2,\dots,b_{i-\xi-k+1}\right), \label{eq:faa}
    \end{align}
    where $\hat{B}_{n,k}(x_1,x_2,\dots,x_{n-k+1})$ denote the ordinary Bell polynomials.
    Note that the equation of $[c^{-i}]$ only depends on the coefficients $b_1,b_2,\dots, b_i$. Thus, if we scale by $p_N-a_Nx^N$ instead of $p_{N}$ the first $N-1$ equations will not be affected. Hence, each $b_i$ and consequently the coefficients $a_1,a_2,\dots,a_i$ are uniquely defined by $i$ differential equations in $(W_\xi(z))_{0\leq \xi\leq i}$ (if there is a valid choice). For $p_5(x)$ and $V_i(z), 0\leq i\leq 5$ we have in particular 
    \[
        q_5(x) = 1-x-3x^3-26x^4-324x^5+O(x^6)
    \]
    and therefore
\begin{align*}
    V_0(z) &= W_0(z)\\
    V_1(z) &= W_1(z)-zW_0'(z) \\
    V_2(z) &= W_2(z)-zW_1'(z)+\frac{z^2}{2!}W_0''(z)\\
    V_3(z) &= W_3(z)-zW_2^{(1)}(z)+\frac{z^2}{2!}W^{(2)}_1(z)-\frac{z^3}{3!}W_0^{(3)}-3zW_0^{(1)}\\
    V_4(z) &= W_4(z)-zW_3^{(1)}(z)+\frac{z^2}{2!}W^{(2)}_2(z)-\frac{z^3}{3!}W_1^{(3)}-3zW_1^{(1)}+\frac{z^4}{4!}W_0^{(4)}(z)\\
    &\quad+\frac{72}{4!}z^2W_0^{(2)}(z)-26zW_0^{(1)}(z)\\
    V_5(z) &= W_5(z)-zW_4^{(1)}(z)+\frac{z^2}{2!}W^{(2)}_3(z)-\frac{z^3}{3!}W_2^{(3)}-3zW_2^{(1)}+\frac{z^4}{4!}W_1^{(4)}(z)\\ &\quad +\frac{72}{4!}z^2W_1^{(2)}(z)-26zW_1^{(1)}(z)-\frac{z^5}{5!}W_0^{(5)}(z)-\frac{3}{2}z^3W_0^{(3)}(z)+26z^2W_0^{(2)}(z)-324zW_0^{(1)}(z).
\end{align*}
It is an easy exercise to check that these expression indeed reduce to the polynomials $Q_i(x)$ in $zC(z)^2$ multiplied by $C(z)$ as mentioned in Proposition~\ref{cor:Vi}.\\
Further, this does not seem to be the limit. Denominators consistently cancel and we are able to choose a coefficient $a_n$ which would satisfy Conjecture~\ref{conj} for $n=6$ and $n=7$ as well. However, since the polynomials grow rapidly, we decided to stop at $n=5$ in this paper and prove Theorem~\ref{thm:expand} as a corollary.

\begin{proof}[Proof of Theorem~\ref{thm:expand}]
    The generating function of the moments of $\Lambda_{f(1/c)}(\mu^c)$ is given by $M_{\mu^c}\left(\frac{z}{f(1/c)}\right)$. Note that
    \[
        f(x)^2 = 1+x+x^2+4x^3+33x^4+386x^5+O(x^6)
    \]
    such that we can expand 
    \[
        M_{\mu^c}\left(\frac{z}{f(1/c)}\right) = M_{\mu^{f^2}}(z) = \sum_{i= 0}^5 \frac{1}{c^i}V_i\left(z^2\right)+\sum_{i\geq 6} \frac{1}{c^i}[c^{-i}]M_{\mu^{f^2}}\left(z\right)
    \]
    where the $V_i(z)$ are given by Proposition~\ref{cor:Vi}. Applying the inversion formula to these functions yield densities of signed measures with null mass. 
\end{proof}

To illustrate why the existence of $p_5(x)$ is surprising,
we observe that if
\[
    \tilde{W}_2(z) :=
    W_2(z) + \frac{z^3 C(z)^7}
    {\left(1 - z C(z)^2 \right)^3}
\]
is given instead of $W_2(z)$, then there is no choice for $[x^2] p(x)$
allowing this magical simplification between numerator and denominator
and the reduction to a polynomial.

This example highlights the difficulty of proving the existence of $P(x)$ in Conjecture~\ref{conj}. A combinatorial approach seems reasonable, but we are not aware of any combinatorial meaning of the generating functions $V_i(z)$, nor do we have a combinatorial interpretation of the differential equations which are satisfied by the generating functions $W_\xi(z)$, except for the equation of $V_1(z)$.
Nevertheless the next theorem sheds partial light on \emph{why} the scaling by $p_5(x)$ results in Proposition~\ref{cor:Vi}.
It shows that keeping the same first two coefficients as in $p_5(x)$
but changing the others gives fractions,
in the expansion in $c^{-1}$, with denominators
that are powers of $1 - z C(z)^2$ that are two less than expected.

\begin{theorem}\label{thm:reduc}
Let $p(x) = \sum_{i\geq 0} x^i$.
Then $\hat{V}_i(z) := [c^{-i}] M_{\mu^{p}}\left(\sqrt{z}\right)$
is a polynomial in $zC(z)^2$ multiplied by $C(z)$ for $i=0,1,2$
and for $i\geq 3$ there exist polynomials $\hat{Q}_i(x)$ such that 
\[
    \hat{V}_i(z) = C(z)\frac{\hat{Q}_i(zC(z)^2)}{\left(1-zC(z)^2\right)^{2i-3}}.
\]
\end{theorem}

We will use the following lemma to prove Theorem~\ref{thm:reduc}.

\begin{lemma}\label{a:lem:derivs}
    Let $W_\xi(z)$ be the generating function of tree walks with excess $\xi$. Then there exist polynomials $q_{\xi,k}(z,y)$  for all $\xi, k \geq 0$ such that the $k$-derivative of $W_\xi(z)$ is given by
    \begin{align*}
        W_0'(z) &= \frac{C(z)^3}{1-zC(z)^2}\\
        W_0^{(k)}(z) &= \frac{(2k-2)!}{(k-1)!}\left(\frac{z^{k-1}C(z)^{4k-1}}{\left(1-zC(z)^2\right)^{2k-1}} + \frac{kz^{k-2}C(z)^{4k-3}}{\left(1-zC(z)^2\right)^{2k-2}}\right)+ \frac{q_{0,k}(z,C(z))}{\left(1-zC(z)^2\right)^{2k-3}}, \qquad k\geq 2\\
        W^{(k)}_\xi(z) &= \frac{(2\xi+2k-2)!}{(\xi+k-1)!\xi!}\left(\frac{z^{4\xi+k-2}C(z)^{8\xi+4k-3}}{\left(1-zC(z)^2\right)^{2\xi+2k-1}}+\frac{(3\xi+k-1)z^{4\xi+k-3}C(z)^{8\xi+4k-5}}{\left(1-zC(z)^2\right)^{2\xi+2k-2}} \right) \\
        &\qquad + \frac{q_{\xi,k}(z,C(z))}{\left(1-zC(z)^2\right)^{2\xi+2k-3}},\qquad \xi,k\geq 1.
    \end{align*}
\end{lemma}
\begin{proof}
    We start with $\xi=0$. Since $W_0(z) = C(z)$, we have
    \[
        W_0'(z) = C'(z) = \frac{C(z)^3}{1-zC(z)^2}
    \]
    and further
    \begin{align*}
        W_0''(z) &= \frac{C(z)^3(2zC(z)C'(z)+C(z)^2)}{\left(1-zC(z)^2\right)^2}+\frac{3\,C(z)^2}{1-zC(z)^2}C'(z) = \frac{2zC(z)^7}{\left(1-zC(z)^2\right)^3} +\frac{4\,C(z)^5}{\left(1-zC(z)^2\right)^2}.
    \end{align*}
    Notice, that we gain at most $+1$ in the exponent of the denominator if we take the derivative of a term in the nominator and $+2$ if we take the derivative in the denominator. Hence, if we are just interested in the terms with the highest exponents in the denominator, we can use induction on this form and check that
    \begin{align*}
        W_0^{(k+1)}(z)
        &= \frac{(2k-2)!}{(k-1)!} \left(\frac{2(2k-1)z^{k}C(z)^{4k+3}}{\left(1-zC(z)^2\right)^{2k+1}} +\frac{(2k-1)z^{k-1}C(z)^{4k+1}}{\left(1-zC(z)^2\right)^{2k}}+\frac{(4k-1)z^{k-1}C(z)^{4k+1}}{\left(1-zC(z)^2\right)^{2k}}\right)\\
        &\qquad +\frac{(2k-2)!}{(k-1)!} \left(\frac{2(2k-2)kz^{k-1}C(z)^{4k+1}}{\left(1-zC(z)^2\right)^{2k}}\right)+ \frac{q_{0,k+1}(z,C(z))}{\left(1-zC(z)^2\right)^{2k-1}}\\
        &= \frac{(2k)!}{k!} \left(\frac{z^{k}C(z)^{4k+3}}{\left(1-zC(z)^2\right)^{2k+1}} +\frac{(k+1)z^{k-1}C(z)^{4k+1}}{\left(1-zC(z)^2\right)^{2k}}\right)+ \frac{q_{0,k+1}(z,C(z))}{\left(1-zC(z)^2\right)^{2k-1}}
    \end{align*}
    where $q_{0,k}(x,y)$ is a bivariate polynomial with non-negative coefficients.
    Similarly, we can compute the main contributions of $W_\xi(z)$ since we know from Theorem~\ref{thm:strucT_k}  that for $\xi \geq 2$
    \[
        W_\xi(z) = \cat(\xi-1)C(z)\left(\frac{\left(zC(z)^2\right)^{4\xi-2}}{\left(1-zC(z)^2\right)^{2\xi-1}}+\frac{(3\xi-1)\left(zC(z)^2\right)^{4\xi-3}}{\left(1-zC(z)^2\right)^{2\xi-2}}\right) + \frac{q_{\xi,0}(zC(z)^2)}{\left(1-zC(z)^2\right)^{2\xi-3}},
    \]
    for some polynomial $q_{\xi,0}(x)$, where $\cat(\xi-1)$ denotes the $(\xi-1)$-th Catalan number $\frac{(2\xi-2)!}{\xi!(\xi-1)!}$. Thus, the expression coincides exactly with our formula for $k=0$.
    It is therefore straight forward to check that
    \begin{align*}
        W^{(k+1)}_\xi(z) &= \cat(\xi-1)\frac{(2\xi+2k-2)!(\xi-1)!}{(\xi+k-1)!(2\xi-2)!}\Bigg(\frac{2(2\xi+2k-1)z^{4\xi+k}C(z)^{8\xi+4k+1}}{\left(1-zC(z)^2\right)^{2\xi+2k+1}} \\
        &\qquad+\frac{(2\xi+2k-1)z^{4\xi+k-1}C(z)^{8\xi+4k-1}}{\left(1-zC(z)^2\right)^{2\xi+2k}}
        +\frac{(8\xi+4k-3)z^{4\xi+k-1}C(z)^{8\xi+4k-1}}{\left(1-zC(z)^2\right)^{2\xi+2k}} \\
        &\qquad+\frac{2(3\xi+k-1)(2\xi+2k-2)z^{4\xi+k-1}C(z)^{8\xi+4k-1}}{\left(1-zC(z)^2\right)^{2\xi+2k}}\Bigg)
        \\
        &\qquad+\frac{q_{\xi,k+1}(z,C(z))}{\left(1-zC(z)^2\right)^{2\xi+2k-1}}.
    \end{align*}
    Now we factor out 
    \[
        2(2\xi+2k-1) = \frac{(2\xi+2k)(2\xi+2k-1)}{(\xi+k)}
    \]
    and confirm that
    \begin{align*}
        &\frac{(2\xi+2k-1)+(8\xi+4k-3)+2(3\xi+k-1)(2\xi+2k-2)}{2(2\xi+2k-1)} = 3\xi+k
    \end{align*}
    which yields the desired result
    \begin{align*}
        W^{(k+1)}_\xi(z) &= \cat(\xi-1)\frac{(2\xi+2k)!(\xi-1)!}{(\xi+k)!(2\xi-2)!}\Bigg(\frac{z^{4\xi+k}C(z)^{8\xi+4k+1}}{\left(1-zC(z)^2\right)^{2\xi+2k+1}} +\frac{(3\xi+k)z^{4\xi+k-1}C(z)^{8\xi+4k-1}}{\left(1-zC(z)^2\right)^{2\xi+2k}}\Bigg)
        \\
        &\qquad+\frac{q_{\xi,k+1}(z,C(z))}{\left(1-zC(z)^2\right)^{2\xi+2k-1}}.
    \end{align*}
    Finally, a quick computation confirms that the same holds for $\xi=1$, even though the formula in the lemma does not hold for $k=0$.
\end{proof}
We are now ready to prove Theorem~\ref{thm:reduc}.
\begin{proof}[Proof of Theorem~\ref{thm:reduc}]
    Since we chose $p(x) = \sum_{k\geq 0}x^k$, its formal inverse equals $q(x) = \sum_{k\geq 0} b_kx^k = 1-x$. That means Equation~(\ref{eq:faa})
    \[
        \left[c^{-i}\right]M_{\mu^p}\left(\sqrt{z}\right) 
        = W_{i}(z)+\sum_{\xi = 0}^{i-1} \frac{1}{(i-\xi)!} \sum_{k=1}^{i-\xi}z^kW_\xi^{(k)}\left(z\right)\hat{B}_{i-\xi,k}\left(b_1,b_2,\dots,b_{i-\xi-k+1}\right),
    \]
    reduces to
    \[
        \left[c^{-i}\right]M_{\mu^p}\left(\sqrt{z}\right) 
        = \sum_{\xi = 0}^{i} \frac{(-1)^{i-\xi}}{(i-\xi)!} z^{i-\xi}W_\xi^{(i-\xi)}\left(z\right).
    \]
    By Lemma~\ref{a:lem:derivs}, we further have
    \begin{align*}
        \left[c^{-i}\right]M_{\mu^p}\left(\sqrt{z}\right) 
        &=  \frac{(-1)^i}{i!}\frac{(2i-2)!}{(i-1)!}\left(\frac{z^{2i-1}C(z)^{4i-1}}{\left(1-zC(z)^2\right)^{2i-1}}+\frac{iz^{2i-2}C(z)^{4i-3}}{\left(1-zC(z)^2\right)^{2i-2}}\right) \\
        &\quad+ \sum_{\xi=1}^i \frac{(-1)^{i-\xi}}{(i-\xi)!\xi!}\frac{(2i-2)!}{(i-1)!}\frac{z^{2\xi+2i-2}C(z)^{4\xi+4i-3}}{\left(1-zC(z)^2\right)^{2i-1}} \\
        &\quad+ \sum_{\xi=1}^i \frac{(-1)^{i-\xi}}{(i-\xi)!\xi!}\frac{(2i-2)!}{(i-1)!}\frac{(2\xi+i-1)z^{2\xi+2i-3}C(z)^{4\xi+4i-5}}{\left(1-zC(z)^2\right)^{2i-2}} \\
        &\quad+ \sum_{\xi=0}^i \frac{(-1)^{i-\xi}}{(i-\xi)!} \frac{z^{i-\xi}q_{\xi,i-\xi}\left(z,C(z)\right)}{\left(1-zC(z)^2\right)^{2i-3}}.
    \end{align*}
    It is straightforward to check for $i=0,1$ that these expressions are polynomials. For $i\geq 2$, we factor out 
    \[
        \frac{(-1)^i(2i-2)!z^{2i-3}C(z)^{4i-5}}{i!(i-1)!(1-zC(z)^2)^{2i-3}}
    \] 
    in the first three lines and focus on 
    \begin{align*}
        &\Bigg(\frac{z^2C(z)^4}{\left(1-zC(z)^2\right)^{2}}+ \sum_{\xi=1}^i \frac{(-1)^{i-\xi}i!}{(i-\xi)!\xi!}\frac{(\xi-1)!}{(2\xi-2)!}\frac{z^{2\xi}C(z)^{4\xi+2}}{\left(1-zC(z)^2\right)^{2}} \\
        &\quad+\frac{izC(z)^2}{\left(1-zC(z)^2\right)} + \sum_{\xi=1}^i \frac{(-1)^{i-\xi}i!}{(i-\xi)!\xi!}\frac{(\xi-1)!}{(2\xi-2)!}\frac{(2\xi+i-1)z^{2\xi-1}C(z)^{4\xi}}{\left(1-zC(z)^2\right)}\Bigg).
    \end{align*}
    If the above is a polynomial in $z,C(z)$, we are done. 
    Now, we substitute $t=zC(z)^2$ to shorten notation and verify that 
    \begin{align*}
        &\frac{t^2}{\left(1-t\right)^2} + \sum_{\xi=1}^i (-1)^{\xi}\binom{i}{\xi}\frac{t^{2\xi+1}}{(1-t)^2} + \frac{it}{1-t}+\sum_{\xi=1}^i (-1)^{-\xi}\binom{i}{\xi} \frac{(2\xi+i-1)t^{2\xi}}{1-t}\\
        &= \frac{t^2+t(1-t^2)^i-t}{(1-t)^2} + \frac{i(t+(1-t^2)^i-1)-(1-t^2)^i+1}{1-t}+\frac{t\frac{d}{dt}(1-t^2)^{i}}{1-t}\\
        &= \frac{t(1+t)(1-t^2)^{i-1}+(i-1)(1-t^2)^i-i(1-t)-2it^2(1-t^2)^{i-1}}{1-t}\\
        &= t(1+t)^2(1-t^2)^{i-2}+(i-1)(1+t)(1-t^2)^{i-1}-2it^2(1+t)(1-t^2)^{i-2}-i
    \end{align*}
\end{proof}

        \section{Computational experiments}
        \label{sec:computational}

As curious as Conjecture~\ref{conj} is from a purely mathematical perspective, the alternative scaling of the matrices of the spectral measure seems to have advantages in the approximation of the limit measure $\bar{\mu}^c$. There are certain important details to take into account though.

Since the $V_i(z)$ in Corollary~\ref{cor:Vi} are polynomials in
$z C(z)^2$ multiplied by $C(z)$,
the evaluation $\frac{1}{z} V_i\left(\frac{1}{z}\right)$
is a polynomial in the Stieltjes transform of the semicircle law.
The inversion formula~(\ref{eq:inversion})
therefore always yields densities of signed measures with zero mass
for these Stieltjes transforms.
In particular, we obtain a sequence of densities $f_i(z)$
from the Stieltjes transforms
$\frac{1}{z}V_i\left(\frac{1}{z}\right)$
for $1\leq i \leq 5$ which are given by
\begin{align*}
    f_0(z) &= \frac{1}{2\pi}\sqrt{4-z^2}\;\mathds{1}_{(-2,2)}(z),\\
    f_{1}(z) &= \frac{1}{2\pi}\big(1-z^2\big)\sqrt{4-z^2}\;\mathds{1}_{(-2,2)}(z),\\
    f_{2}(z) &= \frac{1}{2\pi}\big(1-6z^2+5z^4-z^6\big)\sqrt{4-z^2}\;\mathds{1}_{(-2,2)}(z),\\
    f_{3}(z) &= \frac{1}{2\pi}\left(9-140z^2+358z^4-299z^6+98z^8 -11z^{10}\right)\sqrt{4-z^2}\;\mathds{1}_{(-2,2)}(z),\\
    f_4(z) &= \frac{1}{2\pi}\big(56+1602z^2-8625z^4+16004z^6\\
    &\qquad \quad -13447z^8+5624z^{10}-1143z^{12}+90z^{14}\big)\sqrt{4-z^2}\;\mathds{1}_{(-2,2)}(z),\\
    f_5(z) &=\frac{1}{2\pi}\big(442-17946z^2+171911z^4-574676z^6+904447z^8\\
    &\qquad \quad
    -768354z^{10} +373181z^{12} -103622z^{14} +15298z^{16} -931z^{18} \big)\sqrt{4-z^2}\;\mathds{1}_{(-2,2)}(z). 
\end{align*}

Now, it is easy to see that the coefficients of the polynomial factors of the $f_i(z)$ grow rapidly and that these functions oscillate quite heavily. Hence, there exists a largest integer $t(c)$ depending on $c$ such that 
\[
    \sum_{\xi = 0}^{t(c)} \frac{1}{c^\xi} f_\xi(z)
\]
takes non-negative values on the interval $(-2,2)$ and is therefore the density of a probability measure. Experiments for $c=5, 10, 20$ show that this $t(c)$ seems to be the right scaling for $\bar{\mu}^c$ such that most of the eigenvalues are exactly in the interval $(-2,2)$.
This is reminiscent of divergent asymptotic expansions (see \eg the introduction of \cite{olverasymptotics}). For example, consider Stirling's asymptotic expansion $n! \approx n^n e^{-n} \sqrt{2 \pi n} (s_0 + s_1 n^{-1} + s_2 n^{-2} + \cdots)$ where $(s_0, s_1, s_2, \ldots) = (1, \frac{1}{12}, \frac{1}{288}, \ldots)$. For any $n$, there exists $t(n)$ such that the accuracy of the approximation of order $k$ improves for $k$ from $0$ to $t(n)$, then decreases with $k$.

\begin{table}[ht]
    \centering
    \begin{tabular}{l | c| c| c}
        sample & \multirow{2}{*}{$f_0(z)$} & \multirow{2}{*}{$f_0(z)+\frac{1}{5}f_1(z)$} & \multirow{2}{*}{$f_0(z)+\frac{1}{5}f_1(z)+\frac{1}{25}f_2(z)$}\\
        size &&&\\\hline 
        & \multirow{8}{*}{\includegraphics[width=0.28\textwidth, page=1]{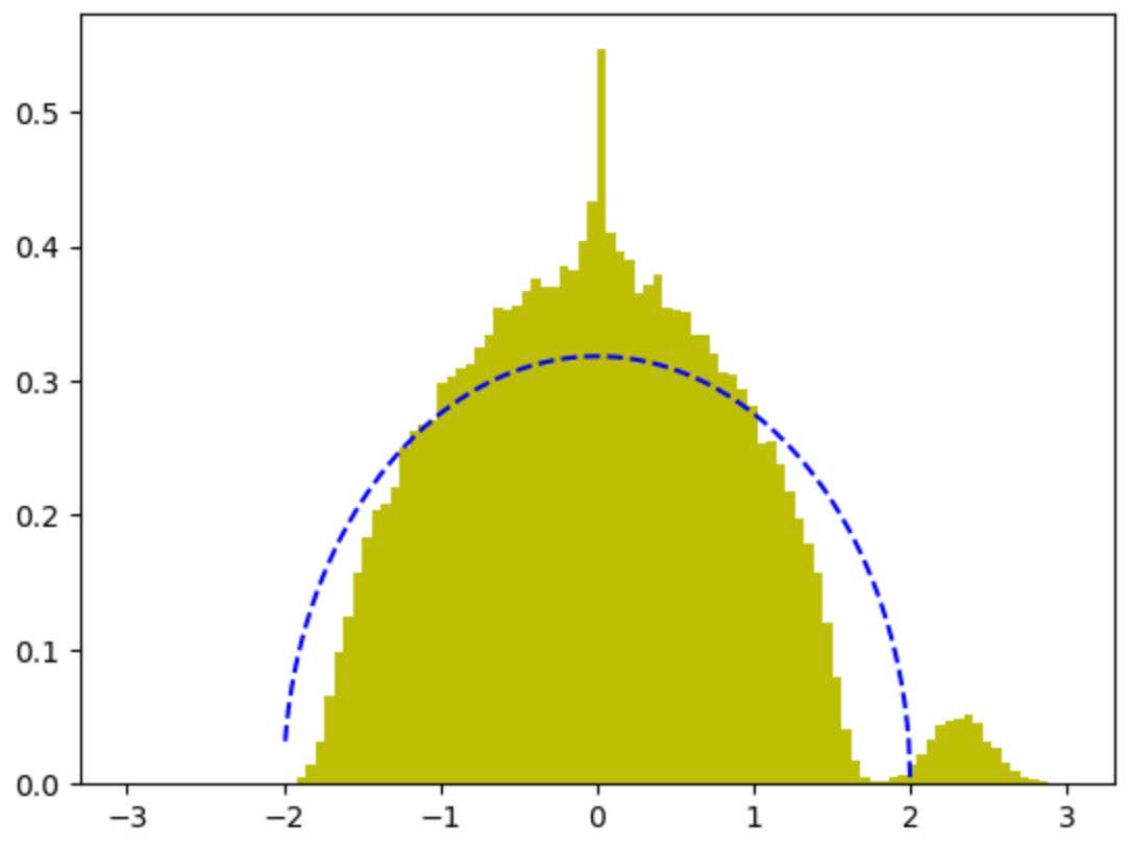}} & \multirow{8}{*}{\includegraphics[width=0.28\textwidth, page=2]{c5plots.pdf}} & \multirow{8}{*}{\includegraphics[width=0.28\textwidth, page=3]{c5plots.pdf}} \\
        &&&\\
        &&&\\
        n=40 &&&\\
        N=2500 & & & \\
        &&&\\
        &&&\\
        &&&\\
        \hline 
        & \multirow{8}{*}{\includegraphics[width=0.28\textwidth, page=4]{c5plots.pdf}} & \multirow{8}{*}{\includegraphics[width=0.28\textwidth, page=5]{c5plots.pdf}} & \multirow{8}{*}{\includegraphics[width=0.28\textwidth, page=6]{c5plots.pdf}} \\
        &&&\\
        &&&\\
        n=200 &&&\\
        N=500 & & & \\
        &&&\\
        &&&\\
        &&&\\
        \hline
        & \multirow{8}{*}{\includegraphics[width=0.28\textwidth, page=7]{c5plots.pdf}} & \multirow{8}{*}{\includegraphics[width=0.28\textwidth, page=8]{c5plots.pdf}} & \multirow{8}{*}{\includegraphics[width=0.28\textwidth, page=9]{c5plots.pdf}} \\
        &&&\\
        &&&\\
        n=1000 &&&\\
        N=100 & & & \\
        &&&\\
        &&&\\
        &&&\\
        \hline
    \end{tabular}
    \caption{Histograms (100 bins) of eigenvalues of $N$ random adjacency matrices of $G(n,5/n)$ compared to the densities $f_0(z), f_1(z)$ and $f_2(z)$.}
    \label{tab:exp_c=5}
\end{table}

\begin{table}[ht]
    \centering
    \begin{tabular}{l | c| c| c}
        sample & \multirow{2}{*}{$\tilde{f}_1(z) = f_0(z)+\frac{1}{10}f_1(z)$} & \multirow{2}{*}{$\tilde{f}_2(z) = \tilde{f}_1(z)+\frac{1}{100}f_2(z)$} & \multirow{2}{*}{$\tilde{f}_3(z) = \tilde{f}_2(z)+\frac{1}{1000}f_3(z)$}\\
        size &&&\\\hline 
        & \multirow{8}{*}{\includegraphics[width=0.28\textwidth]{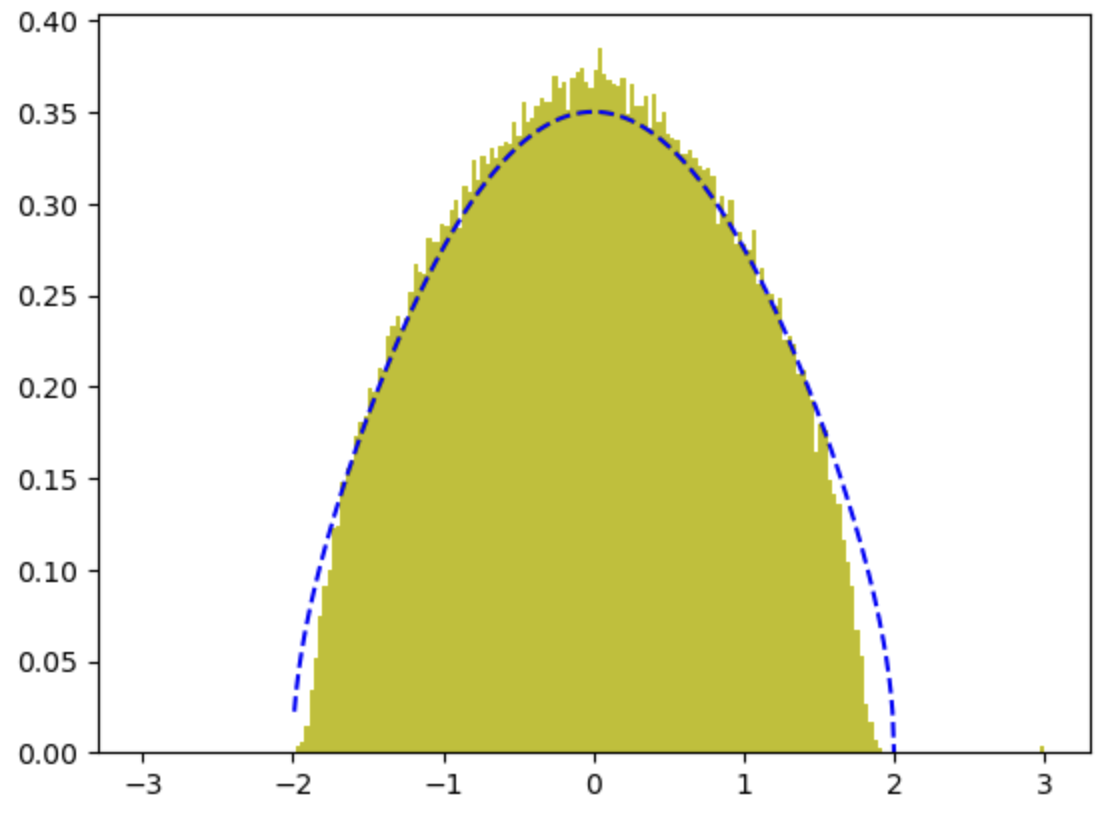}} & \multirow{8}{*}{\includegraphics[width=0.28\textwidth]{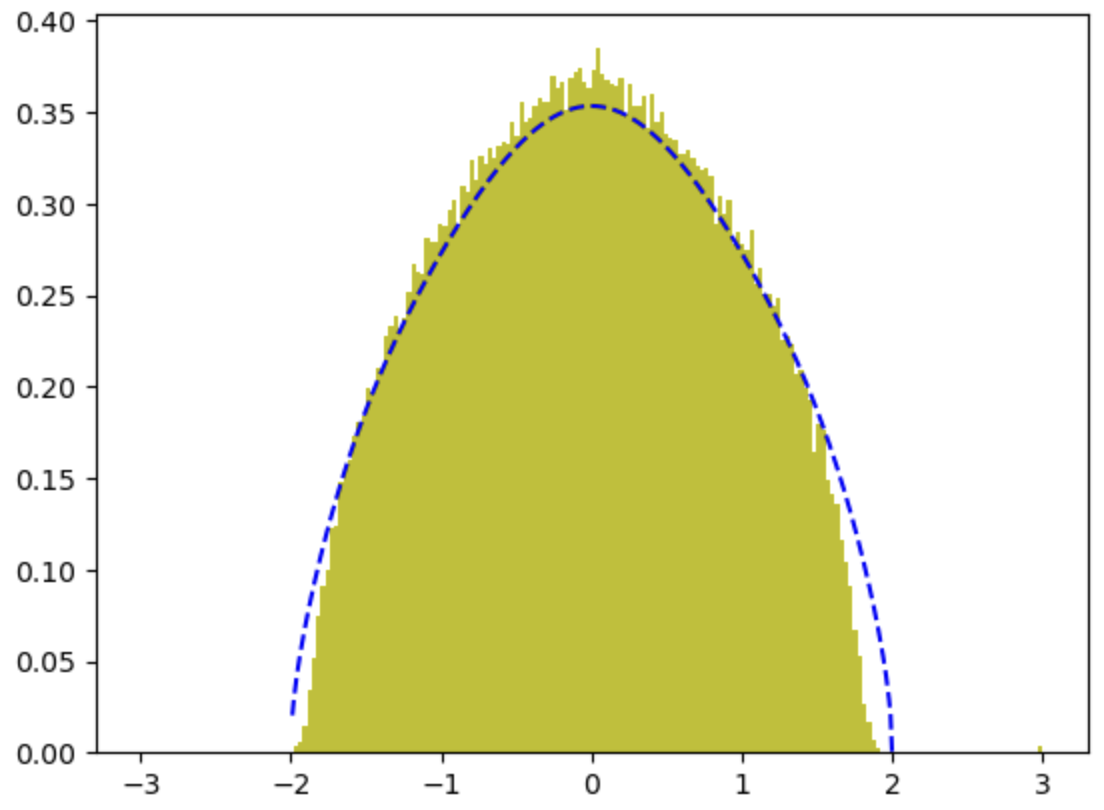}} & \multirow{8}{*}{\includegraphics[width=0.28\textwidth]{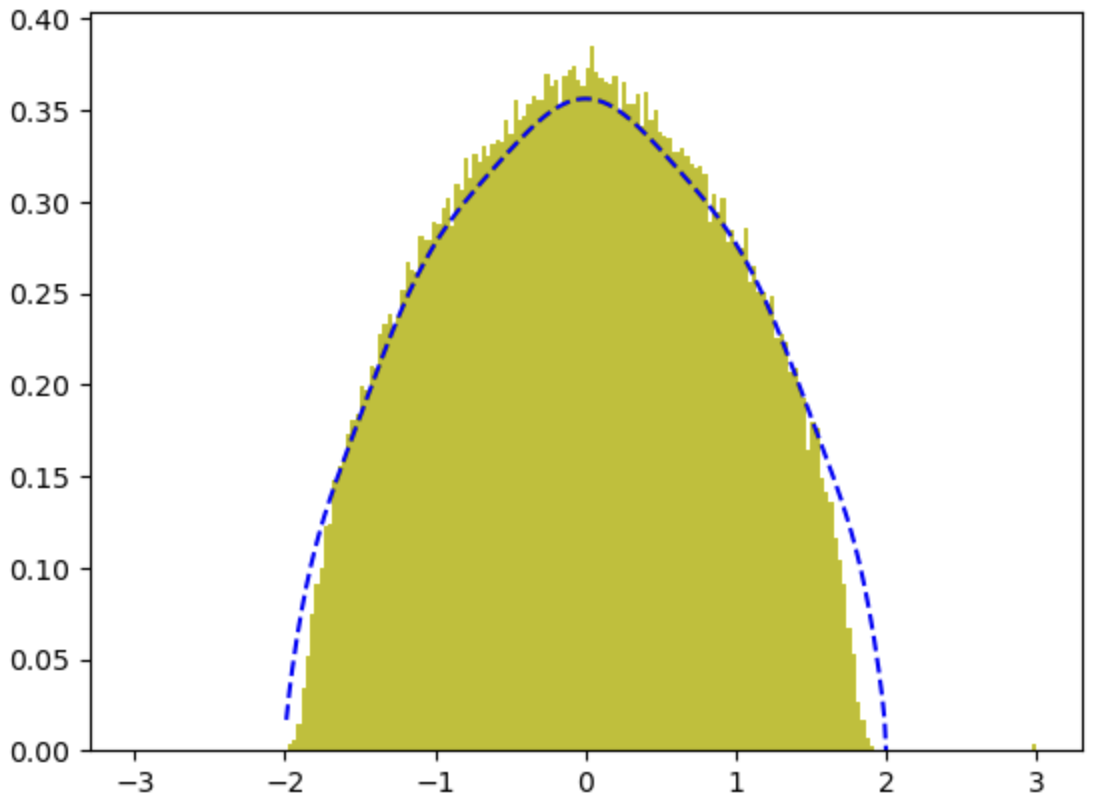}} \\
        &&&\\
        &&&\\
        n=125 &&&\\
        N=800 & & & \\
        &&&\\
        &&&\\
        &&&\\
        \hline 
        & \multirow{8}{*}{\includegraphics[width=0.28\textwidth, page=8]{c10plots.pdf}} & \multirow{8}{*}{\includegraphics[width=0.28\textwidth, page=9]{c10plots.pdf}} & \multirow{8}{*}{\includegraphics[width=0.28\textwidth, page=10]{c10plots.pdf}} \\
        &&&\\
        &&&\\
        n=500 &&&\\
        N=200 & & & \\
        &&&\\
        &&&\\
        &&&\\
        \hline
        & \multirow{8}{*}{\includegraphics[width=0.28\textwidth, page=18]{c10plots.pdf}} & \multirow{8}{*}{\includegraphics[width=0.28\textwidth, page=19]{c10plots.pdf}} & \multirow{8}{*}{\includegraphics[width=0.28\textwidth, page=20]{c10plots.pdf}} \\
        &&&\\
        &&&\\
        n=2000 &&&\\
        N=50 & & & \\
        &&&\\
        &&&\\
        &&&\\
        \hline
    \end{tabular}
    \caption{Histograms (200 bins) of eigenvalues of $N$ random adjacency matrices of $G(n,10/n)$ compared to the densities $\tilde{f}_1(z), \tilde{f}_2(z)$ and $\tilde{f}_3(z)$.}
    \label{tab:exp_c=10}
\end{table}
Further, the densities seem to approximate the histograms of eigenvalues of sampled matrices  quite well. In Table~\ref{tab:exp_c=5}, we can see histograms of random matrices with $p=5/n$. In each row, we sampled $N$ matrices of size $n\times n$ such that we always obtained 100000 eigenvalues. They were scaled by $\sqrt{c(1+1/c)}$ such that we would expect a reasonable approximation by the density $f_0(z)+1/cf_1(z)+1/c^2f_2(z)$. Indeed, in the columns we see the histograms of the eigenvalues in green and the densities given by the approximations of
$f_0(z)$,
$f_0(z) + c^{-1} f_1(z)$
and $f_0(z) + c^{-1} f_1(z) + c^{-2} f_2(z)$.
As $n$ grows, the curve of the latter fits the histogram best.
Another example is illustrated in Table~\ref{tab:exp_c=10}. In this case, $c=10$ and $t(c) = 3$ such that we consider the densities
$f_0(z) + c^{-1} f_1(z)$,
$f_0(z) + c^{-1} f_1(z) + c^{-2} f_2(z)$
and $f_0(z) + c^{-1} f_1(z) + c^{-2} f_2(z) + c^{-3} f_3(z)$.



\bibliography{treewalks.bib}

\end{document}